\documentclass[12pt]{amsart}

\usepackage[utf8]{inputenc}
\usepackage[T1]{fontenc}
\usepackage{amsmath,amssymb,amsthm}
\usepackage{hyperref}

\theoremstyle{plain}
\newtheorem{theorem}{Theorem}[section]
\newtheorem{proposition}[theorem]{Proposition}
\newtheorem{lemma}[theorem]{Lemma}
\newtheorem{corollary}[theorem]{Corollary}
\newtheorem{fact}[theorem]{Fact}

\newtheorem*{thm*}{Theorem}

\theoremstyle{definition}
\newtheorem{definition}[theorem]{Definition}

\theoremstyle{remark}

\newcommand{\AWstar}{$\mathrm{AW}^*$}
\newcommand{\cstar}{$\mathrm{C}^*$}
\newcommand{\wstar}{$\mathrm{W}^*$}
\newcommand{\cU}{\mathcal{U}}

\newcommand{\bA}{\mathbf{A}}
\newcommand{\bM}{\mathbf{M}}
\newcommand{\bN}{\mathbf{N}}
\newcommand{\dt}{d_\tau}
\newcommand{\sa}{\mathrm{sa}}
\newcommand{\Lqt}{L_{\mathrm{qt}}}
\newcommand{\Kqt}{K_{\mathrm{qt}}}

\newcommand{\Th}{\operatorname{Th}}
\newcommand{\Proj}{\operatorname{Proj}}
\newcommand{\ord}{\operatorname{ord}}
\newcommand{\pos}[1]{(#1)_+}
\newcommand{\negp}[1]{(#1)_-}
\newcommand{\bbQ}{\mathbb{Q}}
\newcommand{\bbR}{\mathbb{R}}
\newcommand{\bbC}{\mathbb{C}}
\newcommand{\bbN}{\mathbb{N}}
\newcommand{\bbD}{\mathbb{D}}

\newcommand{\bv}[1]{[\![#1]\!]}
\newcommand{\Tqt}{T_{\mathrm{qt}}}
\newcommand{\Tfac}{T_{\mathrm{fac}}}
\newcommand{\tre}{\tau^{\mathrm{re}}}
\newcommand{\tim}{\tau^{\mathrm{im}}}
\newcommand{\rea}{\operatorname{Re}}

\newcommand{\dist}{\operatorname{dist}}
\newcommand{\dminus}{ 
\buildrel\textstyle\ .\over{\hbox{ 
\vrule height3pt depth0pt width0pt}{\smash-} 
}}

\begin{document}

\title[Model theory of quasitracial \AWstar-algebras]{Model theory of quasitracial \AWstar-algebras and the absoluteness of monotone completeness}
\author{Jananan Arulseelan}

\address{Department of Mathematics, Iowa State University, 396 Carver Hall, 411 Morrill Road, Ames, IA 50011, USA}
\email{jananan@iastate.edu}
\urladdr{https://sites.google.com/view/jananan-arulseelan}


\begin{abstract}
    We examine various applications of logic to $\mathrm{AW}^*$-algebras, particularly the monotone completeness conjecture.  We develop a model theory for $\mathrm{AW}^*$-algebras equipped with a faithful normalized quasitrace.  Using this language, we provide a separable coding which allows us to reduce the finite case of the monotone completeness conjecture to monotone sequential completeness of $\mathrm{II}_1$ $\mathrm{AW}^*$-factors.  Combining an explicit recursive axiomatization with the completeness theorem for continuous logic, we show that this case of the conjecture, and Kaplansky's problem whether every $\mathrm{AW}^*$-factor of type $\mathrm{II}_1$ is a von Neumann algebra, are each equivalent in ZFC to a $\Pi^0_2$ sentence of arithmetic, so that their answers are the same in every $\omega$-model of ZFC.  We also consider real $\mathrm{AW}^*$-algebras.  Here the Baer-$*$-ring notion of an abelian projection must be distinguished from commutativity; following the real-operator-algebra terminology we call the former semi-abelian.  Handelman's matrix theorem implies that a real $\mathrm{AW}^*$-algebra with no nonzero central semi-abelian projection has $\mathrm{AW}^*$ complexification.  We answer a number of other open problems in real \AWstar-algebras.
\end{abstract}

\maketitle

\tableofcontents

\section{Introduction}\label{sec:intro}

Two longstanding problems in the theory of \AWstar-algebras are the Kaplansky conjecture and the monotone completeness conjecture.  The first, due to Kaplansky~\cite{Kap51}, asks whether every \AWstar-factor of type $\mathrm{II}_1$ is a von Neumann algebra.  Equivalently, it asks whether the unique normalized quasitrace on such a factor is linear~\cite{BH82, Haa14}, and it is equivalent to the quasitrace conjecture and to the existence of tracial states on unital stably finite \cstar-algebras~\cite{BH82, MR23}.  The second asks whether every \AWstar-algebra is monotone complete.  Von Neumann algebras and commutative \AWstar-algebras are monotone complete, and Christensen and Pedersen showed that properly infinite \AWstar-algebras are monotone sequentially complete~\cite{CP84}; see~\cite{SW15} for the theory of monotone complete \cstar-algebras.  The related normality conjecture was proved for factors by Sait\^o and Wright~\cite{SW91} and in general by the author and Hanson~\cite{AH26}.

In \cite[Proposition 30]{GS17}, Goldbring and Sinclair observe that the class of structures $(A, \tau)$ where $A$ is a unital \cstar-algebra and $\tau$ is a quasitrace on $A$ is a universally axiomatizable class.  Inspired by this, we develop and study the model theory of quasitracial \AWstar-algebras, that is, of pairs $(A,\tau)$ where $\tau$ is a faithful normalized quasitrace whose closed unit ball is complete for Haagerup’s metric
\[
    \dt(x,y) = \tau\bigl((x - y)^*(x - y)\bigr)^{1/3}.
\]
By results of Haagerup~\cite{Haa14}, these are exactly the finite \AWstar-algebras equipped with a faithful normal normalized quasitrace.  We give applications to the monotone completeness conjecture restricted to finite \AWstar-algebras.  Our first result is an explicit recursive axiomatization.

\begin{thm*}
    There are explicit recursive theories $\Tqt \subseteq \Tfac$ in a recursive continuous language $\Lqt$ such that the models of $\Tqt$ are exactly the structures associated to quasitracial \AWstar-algebras, and the models of $\Tfac$ are exactly the structures associated to finite \AWstar-factors with their canonical quasitraces.
\end{thm*}

This combines Theorems~\ref{thm:elementary} and~\ref{thm:elementary-factors}.  The axioms of $\Tqt$ ask for linearity of the quasitrace only on singly generated \cstar-subalgebras.  Linearity on all abelian \cstar-subalgebras is recovered from the \AWstar-structure of the models, which we obtain from Kaplansky’s characterization of \AWstar-algebras~\cite{Kap51}, and the extension to $2 \times 2$ matrices required of a quasitrace is then automatic~\cite[Corollary~II.1.10]{BH82}.

Using a downward Löwenheim–Skolem argument in this language, we reduce the finite case of the monotone completeness conjecture to separable factors.  The language also gives a uniform reformulation in terms of monotone chains.  For $\eta > 0$ and $\varepsilon \geq 0$, an $(\eta,\varepsilon)$-chain of length $k$ in a quasitracial \AWstar-algebra is a tuple $c_1, \ldots, c_k$ of positive contractions with
\[
    \tau\bigl(\pos{c_i - c_j}\bigr) \leq \varepsilon \qquad \text{and} \qquad \tau\bigl(\pos{c_j - c_i}\bigr) \geq \eta \qquad (i < j).
\]
In a von Neumann algebra the trace bounds the length of such chains, and the following theorem shows that a uniform bound of this kind is equivalent to monotone completeness.

\begin{thm*}
    ZFC proves that the following are equivalent.
    \begin{enumerate}
        \item[(1)] Every finite \AWstar-algebra is monotone complete.
        \item[(2)] Every $\dt$-separable \AWstar-factor of type $\mathrm{II}_1$ is monotone sequentially complete.
        \item[(3)] For every $\eta > 0$ there are $\varepsilon > 0$ and $k$ such that no finite \AWstar-factor contains an $(\eta,\varepsilon)$-chain of length $k$.
    \end{enumerate}
\end{thm*}

This combines Corollary~\ref{cor:separable-reduction}, Proposition~\ref{prop:uniform-chains} and Theorem~\ref{thm:finite-algebras}.  The passage from factors to arbitrary finite \AWstar-algebras uses Ozawa’s Boolean-valued representation of \AWstar-algebras~\cite{Oz86}.

Combining the explicit axioms with the completeness theorem of Ben Yaacov and Pedersen for continuous logic~\cite{BYP10}, we obtain arithmetic encodings of the finite monotone completeness conjecture and of Kaplansky’s problem.

\begin{thm*}
    There are $\Pi^0_2$ sentences $\sigma_{\mathrm{MC}}$ and $\sigma_K$ of arithmetic such that ZFC proves that $\sigma_{\mathrm{MC}}$ holds if and only if every finite \AWstar-algebra is monotone complete, and that $\sigma_K$ holds if and only if every \AWstar-factor of type $\mathrm{II}_1$ is a von Neumann algebra.  Each of them has the form $\forall n\,\pi(n)$, where every instance $\pi(n)$ is a $\Sigma^0_1$ sentence asserting the existence of a proof of a condition from an explicit recursive list.
\end{thm*}

This combines Theorems~\ref{thm:absoluteness} and~\ref{thm:finite-algebras}.

Since an $\omega$-model of ZFC has the standard natural numbers, it agrees with $V$ on arithmetic sentences.

\begin{thm*}
    The finite monotone completeness conjecture and Kaplansky’s problem each have the same answer in every $\omega$-model of ZFC.  In particular, neither answer can be changed by forcing or by passing to an inner model.
\end{thm*}

Consequently, an independence result for either problem cannot be obtained by forcing or inner model arguments, and hypotheses that hold in $L$, such as $V = L$, $\diamondsuit$ or CH, may be assumed without loss of generality in a proof of either answer.  For Kaplansky’s problem, an alternative $\Pi^0_2$ formulation can be extracted from an observation of Milh{\o}j and R{\o}rdam~\cite[Remark~4.7]{MR23}, and the corresponding trace problem was shown by Blackadar and Farah to be absolute between transitive models of ZFC containing all countable ordinals~\cite{BF26}.  We know of no precedent for the encoding of the monotone completeness conjecture.  In~\cite{AH26}, the author and Hanson show by forcing that if some \AWstar-algebra is not monotone complete, then in a generic extension some \AWstar-factor is not monotone complete.  The results here are of a different kind: for finite \AWstar-algebras they identify the conjecture with a fixed arithmetic sentence, whose truth value is the same in every $\omega$-model of ZFC.

Finally, we consider real \AWstar-algebras.  A point which is harmless over the complex numbers becomes essential here: in the Baer $*$-ring literature an algebra is called abelian when all its projections are central, and a projection is abelian when its corner has that property.  Over the reals this need not mean commutative; the quaternion algebra $\mathbb H$ is the basic example.  To avoid ambiguity we call this property \emph{semi-abelian}.  Handelman's real matrix-ring theorem~\cite{Han78}, together with the fixed-point realization of complexification inside $M_2(R)$, gives the following result.

\begin{thm*}
    Let $R$ be a real \AWstar-algebra with no nonzero central semi-abelian projection.  Then the complexification $R_{\bbC}$ is a complex \AWstar-algebra.
\end{thm*}

This is Theorem~\ref{thm:handelman}.  The exceptional factor case is completely explicit: a semi-abelian real \AWstar-factor is $\bbR$ or $\mathbb H$.  Consequently every real \AWstar-factor has a complex \AWstar-complexification of the same refined type, and every real \AWstar-factor is normal.  Chepukhalin and Rakhimov prove the latter conclusion under the additional hypothesis that the complexification is an \AWstar-algebra~\cite[Theorem~2]{CR25}; Corollary~\ref{cor:factor-complex} shows that this hypothesis is automatic.  Combining this with Wright's and Christensen-Pedersen's faithful-state criteria~\cite{Wri76,CP84}, we obtain that every semifinite real \AWstar-factor possessing a faithful state is a real \wstar-factor.  The same complexification theorem also gives the real analogue of Christensen and Pedersen's theorem~\cite{CP84}: every properly infinite real \AWstar-algebra is monotone sequentially complete.

For a general real \AWstar-algebra, let $z$ be the maximal central semi-abelian projection.  Then $(1 - z)R$ always has \AWstar-complexification, while the obstruction is concentrated in $zR$.  We prove
\[
    R_{\bbC}\text{ is \AWstar } \Longleftrightarrow M_n(R)\text{ is real \AWstar\ for some }n \geq 2 \Longleftrightarrow (zR)_{\bbC}\text{ is \AWstar},
\]
and equivalently every nontrivial matrix amplification is real \AWstar.  Thus the obstruction is central and semi-abelian, but it need not be commutative.

The distinction produces a concrete counterexample.  Put
\[
    R = \ell^{\infty}(\bbN,\bbR) + c_0(\bbN,\mathbb H).
\]
Then $R$ is a finite semi-abelian real \AWstar-algebra with $R_{\sa} = Z(R) = \ell^{\infty}(\bbN, \bbR)$, its center is a real \wstar-algebra, but $R$ is not a real \wstar-algebra.  Moreover $R_{\bbC}$ is not even Rickart, and for a fixed imaginary quaternionic unit $I$ the bounded derivation
\[
    \delta((q_n)) = (Iq_n - q_nI)
\]
is outer.  Hence the complexification hypothesis in Ayupov's type $\mathrm{I}$ center criterion~\cite{Ayu04typeI} cannot be removed, and the innerness question for derivations posed by Dadakhodjaev and Rakhimov~\cite{DR21} has a negative answer.

The order-theoretic and quasitracial consequences are nevertheless recoverable because a semi-abelian real \AWstar-algebra $B$ satisfies
\[
    B_{\sa} = Z(B)_{\sa}.
\]
This gives monotone completeness and normality on the semi-abelian band and yields the equivalence, in ZFC, of the real and complex monotone-completeness conjectures.  On the finite side it also supplies the missing semi-abelian part of the center-valued quasitrace theory: scalar $1$-quasitraces factor through the canonical self-adjoint-center-valued quasitrace and extend to $n$-quasitraces for all $n$.  Combining this with Katz's Boolean-valued interpretation~\cite{Katz13} and the complementary complexifiable band, we retain the arithmetic characterization of the real center-valued trace problem and a corrected norm-one projection criterion.

Finally, universal innerness must be discarded, but the local-derivation conclusions survive.  We prove that every local derivation and every $2$-local derivation on a real \AWstar-algebra is a derivation.  Every derivation is inner on the complementary band $(1 - z)R$; all possible outer behavior is concentrated on the central semi-abelian band.

We obtain, in particular, the following corrected summary.

\begin{thm*}
    Let $R$ be a real \AWstar-algebra.
    \begin{enumerate}
        \item[(1)] $R$ is normal.
        \item[(2)] Every local derivation and every $2$-local derivation on $R$ is a derivation.  Derivations need not be inner.
    \end{enumerate}
    Moreover, ZFC proves that every real \AWstar-algebra is monotone complete if and only if every complex \AWstar-algebra is monotone complete.
\end{thm*}

The paper is organized as follows.  Section~\ref{sec:prelims} collects preliminaries on quasitraces, and Section~\ref{sec:ultraproducts} treats quasitracial ultraproducts.  The language and its axioms are described in Section~\ref{sec:language}, which also proves the axiomatization results.  Reflection to separable factors is proved in Section~\ref{sec:reflection}.  The arithmetic encodings and the passage to finite \AWstar-algebras are given in Section~\ref{sec:absolute}.  Section~\ref{sec:real} treats real \AWstar-algebras, including the real Boolean-valued reduction and its trace and quasitrace applications.

Since the initial preparation of this paper, OpenAI has released a manuscript proposing a counterexample to the quasitrace conjecture, and subsequent joint work of the author and Hanson has produced a proposed counterexample to the monotone completeness conjecture. These developments may settle the questions that originally motivated our absoluteness results. Indeed, the latter construction relies on the model-theoretic machinery developed here, particularly the preservation of factoriality under quasitracial ultraproducts and the reflection of monotone incompleteness to separable factors. We retain the original perspective of the paper, which develops these tools and their arithmetic consequences in a general setting.

\subsection*{AI disclosure}

Errors in a previous version were found by GPT 6 Astra.

\section{Preliminaries}\label{sec:prelims}

We use the following conventions for quasitraces.  Let $A$ be a unital complex \cstar-algebra.  A \textbf{$1$-quasitrace} on $A$ is a function $\tau : A \to \bbC$ such that
\begin{enumerate}
    \item[(i)] $\tau(x^*x) = \tau(xx^*) \geq 0$ for every $x \in A$;
    \item[(ii)] $\tau(a + ib) = \tau(a) + i\tau(b)$ for $a,b \in A_{\sa}$;
    \item[(iii)] the restriction of $\tau$ to every abelian \cstar-subalgebra of $A$ is linear.
\end{enumerate}
A \textbf{$2$-quasitrace}, or simply a \textbf{quasitrace}, is a $1$-quasitrace $\tau$ for which there is a $1$-quasitrace $\tau_2$ on $M_2(A)$ satisfying $\tau_2(x\otimes e_{11}) = \tau(x)$.  We call $\tau$ normalized if $\tau(1) = 1$.  On an \AWstar-algebra the distinction between $1$- and $2$-quasitraces disappears: every $1$-quasitrace is a $2$-quasitrace by~\cite[Corollary~II.1.10]{BH82}.

For a unital real \cstar-algebra $R$, a \textbf{scalar $1$-quasitrace} is a function $\tau: R\to\bbR$ such that $\tau(x^*x) = \tau(xx^*) \geq 0$, $\tau(a + b) = \tau(a)$ for $a\in R_{\sa}$ and $b^* = -b$, and $\tau$ is linear on every abelian real \cstar-subalgebra.  This is the convention used in Section~\ref{sec:real}; see~\cite{RR24}.  For $n \geq 1$, such a map is called an \textbf{$n$-quasitrace} if there is a scalar $1$-quasitrace $\tau_n$ on $M_n(R)$ satisfying $\tau_n(x\otimes e_{11}) = \tau(x)$ for every $x \in R$.

\begin{lemma}\label{lem:abelian-comparison}
    Let $\tau$ be a normalized quasitrace on $A$ and put $\dt(x,0) = \tau(x^*x)^{1/3}$.  If $0 \leq h \leq 1$, then
    \[
        \dt(h,0)^3 \leq \tau(h) \leq \dt(h,0)^{3/2}.
    \]
\end{lemma}

\begin{proof}
    The algebra $\text{\cstar}(1,h)$ is abelian, so the quasitrace is a state there.  The first inequality follows from $h^2 \leq h$, and the second from the Cauchy–Schwarz inequality $\tau(h)^2 \leq \tau(h^2)$.
\end{proof}

\begin{fact}[{\cite[Corollary~3.4, Lemma~3.5, Definition~3.6, Lemma~3.7]{Haa14}}]\label{fact:haagerup-metric}
    Let $\tau$ be a faithful normalized quasitrace on a unital \cstar-algebra $A$. Then:
    \begin{enumerate}
        \item[(a)] $\tau$ is order preserving on $A_{\sa}$, and extends uniquely to a quasitrace $\tau_n$ on $M_n(A)$ with $\tau_n(x\otimes e_{11}) = \tau(x)$;
        \item[(b)] $\dt(x,y) = \tau((x - y)^*(x - y))^{1/3}$ is a metric on $A$, with $\dt(x + y,x_0 + y_0) \leq \dt(x,x_0) + \dt(y,y_0)$ and $\dt(x^*,y^*) = \dt(x,y)$;
        \item[(c)] $\dt(xy,0) \leq \|x\|^{2/3} \dt(y,0)$ and $\dt(xy,0) \leq \dt(x,0)\|y\|^{2/3}$;
        \item[(d)] $\tau$ is continuous on $A_{+}$ for $\dt$.
    \end{enumerate}
\end{fact}

Let $(A, \tau)$ be a pair consisting of a unital \cstar-algebra $A$ and a faithful normalized quasitrace.  Every norm ball is $\dt$-closed~\cite[Lemma~3.8]{Haa14}.  If the unit ball is $\dt$-complete, then $A$ is an \AWstar-algebra and $\tau$ is normal \cite[Lemma 3.9]{Haa14}. Conversely, if $A$ is a finite \AWstar-algebra and $\tau$ is faithful and normal, then every norm ball is $\dt$-complete \cite[Proposition 3.10]{Haa14}.

\begin{definition}\label{def:quasitracial}
    A \textbf{quasitracial \AWstar-algebra} is a pair $(A,\tau)$, where $A$ is a unital \cstar-algebra and $\tau$ is a faithful normalized quasitrace on $A$ whose closed unit ball is complete for $\dt$.
\end{definition}

\begin{fact}[{\cite[Lemma~3.9]{Haa14}}]\label{fact:completion}
    If $(A, \tau)$ is a quasitracial \AWstar-algebra, then $A$ is an \AWstar-algebra and $\tau$ is normal.
\end{fact}

Such algebras are finite: if $v^*v = 1$, then $vv^*$ and $1 - vv^*$ lie in an abelian \cstar-subalgebra, so $\tau(1 - vv^*) = 1 - \tau(vv^*) = 1 - \tau(v^*v) = 0$ and $vv^* = 1$.

\begin{fact}[{\cite[Theorem~II.1.7]{BH82}, \cite[Proposition~3.10 and proof of Corollary~5.13]{Haa14}, \cite[Corollary~7]{Wri76}}]\label{fact:finite-factors}
    Finite \AWstar-factors are of type $\mathrm{I}_n$ or $\mathrm{II}_1$.  Every finite \AWstar-factor $M$ admits a unique normalized quasitrace $\tau_{M}$, which is faithful and normal.  This quasitrace agrees with the dimension function on projections and makes $(M, \tau_{M})$ into a quasitracial \AWstar-algebra.  Moreover, $\tau_{M}$ is a trace if and only if $M$ is a von Neumann algebra.
\end{fact}

\section{Quasitracial Ultraproducts}\label{sec:ultraproducts}

We use Haagerup's ultraproduct construction \cite[Section 4]{Haa14} extended to arbitrary families of quasitracial \AWstar-algebras.  Blackadar and Handelman consider an alternative ultrapower construction based on the operator norm \cite{BH82}.  Goldbring and Sinclair later extend that ultrapower construction to an ultraproduct construction which underlies their axiomatizability result \cite{GS17}.

\begin{lemma}\label{lem:ultraproduct}
    Let $(A_i, \tau_i)_{i \in I}$ be a family of quasitracial \AWstar-algebras and $\cU$ an ultrafilter on $I$.  Define
    \[
        J_{\cU} = \{ (x_i) \in \ell^{\infty}((A_i)_{i\in I}) \ : \ \lim_{i \to \cU} \tau_i(x_i^* x_i) = 0\}.
    \]
    Then $J_{\cU}$ is a norm-closed two-sided ideal of $\ell^{\infty}((A_i)_{i\in I})$.  Furthermore, 
    \[
        \prod_{\cU}(A_i, \tau_i) : = \ell^{\infty}((A_i)_{i\in I})/J_{\cU}
    \]
    with
    \[
        \tau_{\cU}([x_i]) = \lim_{i \to \cU} \tau_i(x_i)
    \]
    is a quasitracial \AWstar-algebra.  Its norm balls are the metric space ultraproducts of the corresponding norm balls of the $(A_i,\tau_i)$ equipped with the corresponding restriction of $\dt$. 
\end{lemma}

\begin{proof}
    Straightforward $2 \times 2$ matrix calculations show that $(x_i) \longmapsto \lim_{i \to \cU} \tau_i(x_i)$ defines a normalized quasitrace on $\ell^{\infty}((A_i)_{i\in I})$.  Hence~\cite[Proposition~3.2]{Haa14} gives us the ideal and quotient assertions and the induced quasitrace $\tau_{\cU}$.  On every fixed norm ball,
    \[
        d_{\tau_{\cU}}([x_i],[y_i]) = \lim_{i \to \cU} d_{\tau_i}(x_i,y_i),
    \]
    and the quotient map carries the closed ball onto the closed ball.  Thus the norm balls of the quotient are the corresponding metric ultraproducts.  Metric ultraproducts of complete metric spaces are complete, so the balls are complete, and Fact~\ref{fact:completion} finishes the proof.
\end{proof}

\begin{definition}
    We call $(\prod_{\cU}(A_i, \tau_i), \tau_{\cU})$ the \textbf{quasitracial ultraproduct} of the family $(A_i, \tau_i)$.  When $A_i = A$ and $\tau_i = \tau$ for all $i \in I$, we refer to the \textbf{quasitracial ultrapower} of $(A, \tau)$.
\end{definition}

\section{The quasitracial language and its axioms}\label{sec:language}

We use continuous first-order logic for many-sorted metric structures as in~\cite{BBHU08}.  Each sort carries a metric bounded by a declared constant, and each function and predicate symbol comes with a declared modulus of uniform continuity, which every structure is required to respect.  Besides the connectives of~\cite{BBHU08} we use finitely many explicit continuous connectives: rational polynomials restricted to bounded intervals, $t \mapsto |t|$, $t \mapsto t^{3/4}$, $\max$, $\min$ and $\dminus$.  Each of these is a uniform limit of a recursive sequence of connectives built from the connectives of~\cite{BYP10}, with explicit error bounds.  Consequently every formula or sentence we write is the uniform limit, uniformly over all structures, of a recursive sequence of formulas or sentences in the syntax of~\cite{BYP10}, with a recursive error bound.  For every displayed formula $\varphi$ below we fix, once and for all, basic-syntax approximants $\varphi^{[m]}$ satisfying
\[
    \sup_M\bigl|\varphi^M-(\varphi^{[m]})^M\bigr|\leq2^{-m}.
\]
Each axiom of the form $\varphi = 0$ may therefore be replaced by a recursive family of basic-syntax conditions with errors tending effectively to zero, without changing the class of models.  We use these replacements tacitly.

When we write an axiom as an equation $\varphi = \psi$ between formulas with free variables, we mean the sentence $\sup_{\bar x} |\varphi - \psi| = 0$, and an inequality $\varphi \leq \psi$ means $\sup_{\bar x} (\varphi \dminus \psi) = 0$.  For a sort $D$, a map $\iota : D \to D'$ and an element $w$ of $D'$, we write $\dist(w, \iota(D))$ for the formula $\inf_{z \in D} d'(w, \iota(z))$.

Let $\Lqt$ be the continuous many-sorted language consisting of symbols for the following, where $n, m$ range over positive integers.
\begin{itemize}
    \item[(L1)]  There are sorts for the operator norm balls $D_n$ of $A$, equipped with metrics $d_n$ bounded by $(2n)^{2/3}$.
    \item[(L2)]  For $n \leq m$, there are inclusions $\iota_{n, m} : D_n \to D_m$, whose modulus of uniform continuity is given by $\delta(\epsilon) = \epsilon$.
    \item[(L3)] There are constant symbols $0$ and $1$ in $D_1$.
    \item[(L4)]  There are function symbols for addition $+ : D_n \times D_m \to D_{n + m}$ with uniform modulus of continuity $\delta(\epsilon) = \dfrac{\epsilon}{2}$.
    \item[(L5)]  There are function symbols for multiplication $\times : D_n \times D_m \to D_{nm}$ with uniform modulus of continuity $\delta(\epsilon) = \dfrac{\epsilon}{m^{2/3} + n^{2/3}}$.
    \item[(L6)]  There are function symbols for adjoints $* : D_n \to D_n$ with uniform modulus of continuity $\delta(\epsilon) = \epsilon$.
    \item[(L7)]  There are function symbols for scalar multiplication $\lambda : D_n \to D_{\max(1, \lceil |\lambda| n \rceil)}$ for each $\lambda \in \bbQ(i)$ with uniform modulus of continuity $\delta(\epsilon) = \dfrac{\epsilon}{|\lambda|^{2/3}}$ for $\lambda \not= 0$ and $\delta(\epsilon) = \epsilon$ for $\lambda = 0$.
    \item[(L8)]  There are predicate symbols $\tau^{\mathrm{re}}_n, \tau^{\mathrm{im}}_n : D_n \to [-n,n]$ for the real and imaginary parts of the quasitrace, each with modulus of uniform continuity given by $\delta(\epsilon) = \left(\dfrac{\epsilon}{4\sqrt{n}}\right)^{4/3}$.
\end{itemize}

We write $\tau(x) = \tau^{\mathrm{re}}(x) + i\tau^{\mathrm{im}}(x)$, and we write $\operatorname{Re} x = \frac{1}{2}(x + x^*)$ and $\operatorname{Im} x = \frac{1}{2i}(x - x^*)$, which are terms. Subscripts for sorts are omitted when clear from context.

A quasitracial AW$^*$-algebra $(A, \tau)$ gives an $\Lqt$-structure $\bA$ in the obvious way: $D_n$ is the closed ball of radius $n$ in $A$ with the metric $\dt$, and the predicates are the real and imaginary parts of $\tau$.

The theory $\Tqt$ consists of the following axiom schemes.

\begin{enumerate}
    \item[(Q1)] Axioms expressing that the inclusions are isometric and compatible:
    \[
        d_m(\iota_{n,m} x, \iota_{n,m} y) = d_n(x,y), \qquad
        \iota_{m,p}\iota_{n,m} = \iota_{n,p}.
    \]
    All function and relation symbols commute with the inclusions, for instance $\iota(x + y) = \iota x + \iota y$.
    \item[(Q2)] Axioms for a unital $*$-algebra over $\mathbb{Q}(i)$. More precisely, the identities of a unital $*$-algebra over $\mathbb{Q}(i)$, each stated after including both sides into a common sort: associativity and commutativity of addition, associativity of multiplication, distributivity, $0 + x = x$,
    $x + (-1)x = 0$, $1x = x1 = x$, $(xy)^* = y^*x^*$, $(x + y)^* = x^* + y^*$,
    $x^{**} = x$, $(\lambda x)^* = \bar\lambda x^*$, $\lambda(\mu x) = (\lambda\mu)x$,
    $(\lambda + \mu)x = \lambda x + \mu x$, $\lambda(x + y) = \lambda x + \lambda y$ and $\lambda(xy) = (\lambda x)y = x(\lambda y)$.
    \item[(Q3)] Axioms expressing convexity. More precisely, for every $t \in \mathbb{Q} \cap [0,1]$, we have the axiom
    \[
        \sup_{x,y \in D_1} \operatorname{dist}\bigl(tx + (1-t)y, \iota_{1,2}(D_1)\bigr) = 0.
    \]
    \item[(Q4)] Axioms expressing truncation. For every $m$, we have the axiom
    \[
        \sup_{x \in D_m} \Bigl(\operatorname{dist}\bigl(x, \iota_{1,m}(D_1)\bigr) \dminus 2\operatorname{dist}\bigl(x^*x, \iota_{1,m^2}(D_1)\bigr)\Bigr) = 0.
    \]
    \item[(Q5)] Axioms for the quasitrace identities:
    \[
        \tau(1) = 1, \quad \tau(x^*x) = \tau(xx^*), 
    \]
    and
    \[
        \tau(x) = \tau(\operatorname{Re} x) + i\tau(\operatorname{Im} x), \quad \tau^{\mathrm{im}}(\operatorname{Re} x) = 0.
    \]
    \item[(Q6)] Axioms expressing single-generator linearity. For every $n$, all rational polynomials $p, q \in \mathbb{Q}[t]$ and every $r \in \mathbb{Q}$, with $h = \operatorname{Re} x$ for $x \in D_n$, we have the axioms
    \[
        \tau^{\mathrm{re}}\bigl(p(h) + q(h)\bigr) = \tau^{\mathrm{re}}(p(h)) + \tau^{\mathrm{re}}(q(h)),
    \]
    \[
        \tau^{\mathrm{re}}(r\,p(h)) = r\,\tau^{\mathrm{re}}(p(h)).
    \]
    \item[(Q7)] Axioms identifying the metric:
    \[
        d_n(x,y)^3 = \tau^{\mathrm{re}}\bigl((x - y)^*(x - y)\bigr).
    \]
    \item[(Q8)] Axioms expressing the module inequality. For all $n, m$, we have the axiom
    \[
        \sup_{x \in D_n,\, y \in D_m} \bigl(\tau^{\mathrm{re}}(y^*x^*xy) \dminus n^2\tau^{\mathrm{re}}(y^*y)\bigr) = 0.
    \]
\end{enumerate}

Note that Axiom (Q6) imposes linearity of $\tau$ only on polynomials in a single self-adjoint element.  It is not obvious to us that the set of commuting pairs is definable for $\dt$; for traces this is the question of whether pairs that almost commute in a $2$-norm are close to commuting pairs.  We recover linearity on all abelian \cstar-subalgebras from the \AWstar-structure of the models.  Axiom (Q8) plays the role of the axiom $\|xy\|_2 \leq \|x\|\|y\|_2$ for tracial von Neumann algebras~\cite{FHS14}, and Axiom (Q4) is a quantitative form of the implication $\|x^*x\| \leq 1 \Rightarrow \|x\| \leq 1$, from which the \cstar-identity is derived in Lemma~\ref{lem:cstar-structure}.  No axioms about $M_2(A)$ are needed: on an \AWstar-algebra, every function satisfying (i)–(iii) of the definition of a quasitrace extends to $M_2(A)$~\cite[Corollary~II.1.10]{BH82}.

\begin{lemma}\label{lem:monotone-calculus}
    Let $(A,\tau)$ be a quasitracial \AWstar-algebra, let $a \leq b$ be self-adjoint, and let $f : \bbR \to [0,\infty)$ be continuous and nondecreasing with $f = 0$ on $(-\infty,0]$.  Then $\tau(f(a)) \leq \tau(f(b))$.
\end{lemma}

\begin{proof}
    First we show that $E_{(t,\infty)}(a) \precsim E_{(t,\infty)}(b)$ for every $t \in \bbR$; compare~\cite[Corollary~II.1.3]{BH82}.  Set $e = E_{(t,\infty)}(a)$, $f_0 = E_{(-\infty,t]}(b)$ and $r = e \wedge f_0$.  Since $r \leq e$ we have $r(a - t)r \geq 0$, and since $r \leq f_0$ we have $r(b - t)r \leq 0$.  Thus $a \leq b$ forces $r(a - t)r = 0$.  With $c = ((a - t)e)^{1/2}$ we get $(cr)^*(cr) = 0$, hence $(a - t)r = 0$ and $r \leq E_{\{t\}}(a)$, which is orthogonal to $e$.  Thus $r = 0$, and Kaplansky's parallelogram law~\cite[Theorem~5.4]{Kap51} gives $e = e - e \wedge f_0 \sim e \vee f_0 - f_0 \leq E_{(t,\infty)}(b)$.

    For $c \geq 0$, the quasitrace is linear and norm continuous on the abelian \AWstar-subalgebra generated by $c$ by~\cite[Proposition~II.1.6]{BH82}, so by step-function approximation we have
    \[
        \tau(c) = \int_0^{\|c\|} \tau\bigl(E_{(s,\infty)}(c)\bigr)\, ds.
    \]
    For $s \geq 0$ put $t_s = \sup\{ t \ : \ f(t) \leq s \}$.  Since $f$ is continuous and nondecreasing, we have $E_{(s,\infty)}(f(a)) = E_{(t_s,\infty)}(a)$, and similarly for $b$.  Here, $t_s$ is allowed to take the value $+\infty$, with the convention that $E_{(+\infty,\infty)}(x)=0$.  Since equivalent projections have the same quasitrace and $\tau$ is order preserving, the first paragraph gives $\tau(E_{(s,\infty)}(f(a))) \leq \tau(E_{(s,\infty)}(f(b)))$ for every $s$, and an integration finishes the proof.
\end{proof}

\begin{proposition}\label{prop:correctness1}
    For every quasitracial \AWstar-algebra $(A,\tau)$, the structure $\bA$ is an $\Lqt$-structure and a model of $\Tqt$.
\end{proposition}

\begin{proof}
    The moduli for addition, multiplication and the adjoint follow from Fact~\ref{fact:haagerup-metric}(b) and (c), and the modulus for scalars from homogeneity of $\tau$ on positive elements.  For self-adjoint $h,k$ of norm at most $n$, the elements $h + n$ and $k + n$ are positive with quasitrace at most $2n$, and $\tau$ is linear on $\text{\cstar}(1,h)$ and on $\text{\cstar}(1,k)$, so~\cite[Lemma~3.7(4)]{Haa14} gives
    \[
        |\tau(h) - \tau(k)| = |\tau(h + n) - \tau(k + n)| \leq 2\sqrt{(2n)}\,\dt(h,k)^{3/4}.
    \]
    Since $\dt(\rea x, \rea y) \leq 2^{1/3}\dt(x,y)$ and $\tre(x) = \tau(\rea x)$, we obtain $|\tre(x) - \tre(y)| \leq 2^{7/4}\sqrt{n}\, \dt(x,y)^{3/4}$, and similarly for $\tim$.

    The sorts $D_n$ are complete since $(A,\tau)$ is quasitracial.

    Axioms (Q1), (Q2), (Q3), (Q5), (Q6) and (Q7) hold by construction, and (Q8) holds because $y^*x^*xy \leq \|x\|^2y^*y$ and $\tau$ is order preserving.  For (Q4), let $\|x\| \leq m$, put $h = x^*x$, let $w$ lie in the unit ball and put $k = \rea w$.  Let $g(t) = \min(1, t^{-1/2})$ for $t > 0$ and $g(0) = 1$, and put $y = xg(h)$.  Then $\|y\| \leq 1$ and
    \[
        \dt(x,y)^3 = \tau\bigl(h(1 - g(h))^2\bigr) = \tau\bigl(\pos{h^{1/2} - 1}^2\bigr) \leq \tau\bigl(\pos{h - 1}^2\bigr).
    \]
    Since $\|k\| \leq 1$ we have $h - 1 \leq h - k$, and Lemma~\ref{lem:monotone-calculus} with $f(s) = \pos{s}^2$ gives
    \[
        \tau\bigl(\pos{h - 1}^2\bigr) \leq \tau\bigl(\pos{h - k}^2\bigr) \leq \tau\bigl((h - k)^2\bigr) = \dt(h,k)^3.
    \]
    As $h - k = \rea(h - w)$, we get $\dt(h,k) \leq 2^{1/3}\dt(h,w)$, and taking the infimum over $w$ gives (Q4).
\end{proof}

\begin{theorem}\label{thm:elementary}
    The models of $\Tqt$ are exactly the structures $\bA$ associated to quasitracial \AWstar-algebras $(A,\tau)$, and $(A,\tau)$ is determined by $\bA$ up to isomorphism.  In particular, the class $\Kqt$ of quasitracial \AWstar-algebras is axiomatized by the explicit recursive theory $\Tqt$.
\end{theorem}

Proposition~\ref{prop:correctness1} gives one inclusion.  For the other, fix a model $\mathfrak{M}$ of $\Tqt$.  By (Q1) the inclusions are isometric embeddings, and we let $A$ be the direct limit of the sorts $D_n$.  By (Q2), $A$ is a unital $*$-algebra over $\bbQ(i)$, the predicates define a function $\tau : A \to \bbC$, and each $D_n \subseteq A$ is complete for $d$.  Since inclusions are isometric and $D_1$ is complete, $\iota(D_1)$ is closed in every $D_m$.  For rational $q > 0$ put $B(q) = \{x \in A \ : \ q^{-1}x \in D_1\}$ and define
\[
    \|x\| = \inf\{q \in \bbQ_{>0} \ : \ x \in B(q)\}.
\]
\begin{lemma}\label{lem:cstar-structure}
    After extending scalars from $\bbQ(i)$ to $\bbC$ by continuity, $(A, \|\cdot\|)$ is a unital \cstar-algebra, $D_n = \{x \in A \ : \ \|x\| \leq n\}$ for every $n$, and $\tau(x^*x) \leq \|x\|^2$ for every $x \in A$.  In particular $d(x,y) \leq \|x - y\|^{2/3}$.
\end{lemma}

\begin{proof}
    Scalar multiplication sends $D_n$ into $D_1$ for $\lambda = 1/n$ and $D_1$ into $D_n$ for $\lambda = n$, so $D_n = B(n)$.  Each $B(q)$ is closed in every sort containing it, scalars of modulus at most $1$ preserve $D_1$, and hence $\|x\| \leq q$ if and only if $x \in B(q)$, for rational $q$.  In particular $D_n = \{\|x\| \leq n\}$.

    Typing of the function symbols gives $\|xy\| \leq \|x\|\|y\|$, $\|x^*\| = \|x\|$, $\|ix\| = \|x\|$ and $\|rx\| = r\|x\|$ for rational $r > 0$.  If $s > \|x\|$ and $t > \|y\|$ are rational, then (Q3) gives
    \[
        \frac{x + y}{s + t} = \frac{s}{s + t} \cdot \frac{x}{s} + \frac{t}{s + t} \cdot \frac{y}{t} \in D_1,
    \]
    so $\|x + y\| \leq \|x\| + \|y\|$.  If $x \in D_1$, then (Q8) with $y = 1$ gives $\tau(x^*x) \leq 1$, and homogeneity from (Q6) gives $\tau(x^*x) \leq \|x\|^2$ in general.  In particular $\|x\| = 0$ implies $x = 0$.

    For the \cstar-identity, let $c > 0$ be rational with $c^2 > \|x^*x\|$ and put $x' = x/c$.  Then $x'^*x' \in D_1$, so (Q4) and closedness give $x' \in D_1$, and hence $\|x\|^2 \leq \|x^*x\|$.  For $\lambda \in \bbQ(i)$ this gives $\|\lambda x\|^2 = |\lambda|^2\|x\|^2$, since $|\lambda|^2$ is rational.

    A norm-Cauchy sequence lies in some $D_N$ and is $d$-Cauchy there, since $d(x_j,x_k)^3 \leq \|x_j - x_k\|^2$.  If $x$ is its $d$-limit and $x_j - x_k \in B(\varepsilon)$ for $j,k \geq J$, closedness of $B(\varepsilon)$ gives $\|x_j - x\| \leq \varepsilon$ for $j \geq J$.  So $A$ is complete, scalar multiplication extends to $\bbC$ by continuity, and $\|1\| = 1$ since $1 \not= 0$.
\end{proof}

\begin{lemma}\label{lem:tau-properties}
    The function $\tau$ has the following properties.
    \begin{enumerate}
        \item[(a)] $\tau(1) = 1$, $\tau(x^*x) = \tau(xx^*) \geq 0$, and $\tau(x^*x) = 0$ only for $x = 0$.
        \item[(b)] $\tau(a + ib) = \tau(a) + i\tau(b)$ for self-adjoint $a,b$.
        \item[(c)] $\tau$ is norm continuous on bounded subsets of $A$.
        \item[(d)] For every self-adjoint $h \in A$, the restriction of $\tau$ to $\text{\cstar}(1,h)$ is a state.
        \item[(e)] $\tau(y^*x^*xy) \leq \|x\|^2\tau(y^*y)$ and $\tau(xyy^*x^*) \leq \|y\|^2\tau(xx^*)$ for all $x,y \in A$.
    \end{enumerate}
\end{lemma}

\begin{proof}
    Parts (a) and (b) follow from (Q5) and (Q7), and (c) from the declared modulus of the predicates and Lemma~\ref{lem:cstar-structure}.  For (d), by (Q5) and (Q6) the quasitrace is real-valued and $\bbQ$-linear on $\bbQ[h]$.  For $a \in \bbQ[h]$ and $t \in \bbQ$ we have $0 \leq \tau((a - t)^2) = \tau(a^2) - 2t\tau(a) + t^2$, so $\tau(a)^2 \leq \tau(a^2) \leq \|a\|^2$.  Thus $|\tau(a)| \leq \|a\|$ on the dense subspace $\bbQ[h]$ of $\text{\cstar}(1,h)_{\sa}$, so by (c) the restriction of $\tau$ to $\text{\cstar}(1,h)_{\sa}$ is the continuous $\bbR$-linear extension $\phi$ of its restriction to $\bbQ[h]$.  Since $\|\phi\| \leq 1 = \phi(1)$, the functional $\phi$ is positive, and with (b) $\tau$ is a state on $\text{\cstar}(1, h)$.  Part (e) follows from (Q8) by scaling, applied to $(y^*,x^*)$ for the second inequality.
\end{proof}

\begin{lemma}\label{lem:borel-calculus}
    Let $h \in A$ be self-adjoint.  There is an injective unital $*$-homomorphism $\Phi_h : L^{\infty}(\mu_h) \to A$ extending the continuous functional calculus, with $\|\Phi_h(f)\| \leq \|f\|_{\infty}$ and $\tau(\Phi_h(f)) = \int f \, d\mu_h$.  Every element of $A$ commuting with $h$ commutes with the range of $\Phi_h$.
\end{lemma}

\begin{proof}
    Let $f$ be bounded Borel and let $c \geq \|f\|_{\infty}$ be rational.  Choose continuous $f_j$ on $\sigma(h)$ with $|f_j| \leq c$ and $f_j \to f$ in $L^2(\mu_h)$.  Then $f_j(h) \in B(c)$ and $d(f_j(h), f_k(h))^3 = \int |f_j - f_k|^2 \, d\mu_h \to 0$, so $f_j(h)$ converges in $d$ to an element $\Phi_h(f) \in B(c)$ which does not depend on the choice of $(f_j)$.  If also $g$ is bounded Borel with approximants $g_j$ as above, then $\|f_jg_j - fg\|_{L^2(\mu_h)} \leq c\|g_j - g\|_{L^2(\mu_h)} + \|g\|_{\infty}\|f_j - f\|_{L^2(\mu_h)} \to 0$, while $f_j(h)g_j(h) \to \Phi_h(f)\Phi_h(g)$ in $d$ because multiplication is $d$-continuous on bounded sets; hence $\Phi_h(fg) = \Phi_h(f)\Phi_h(g)$.  The same argument for sums and adjoints, together with continuity of $\tau$, shows that $\Phi_h$ is a $*$-homomorphism with $\tau(\Phi_h(f)) = \int f \, d\mu_h$, and $\tau(\Phi_h(f)^*\Phi_h(f)) = \int |f|^2 \, d\mu_h$ gives injectivity.  The norm bound holds since $c$ was arbitrary, and the commutation statement passes to the limit.
\end{proof}

\begin{lemma}\label{lem:orthogonal-sups}
    Every family of mutually orthogonal projections in $A$ has a least upper bound in the projection lattice of $A$.
\end{lemma}

\begin{proof}
    Let $(p_i)$ be orthogonal nonzero projections.  Any finitely many of them, $p_{i_1}, \ldots, p_{i_N}$, lie in $\text{\cstar}(1, \sum_j j \,p_{i_j})$, so Lemma~\ref{lem:tau-properties}(d) gives $\sum_j \tau(p_{i_j}) = \tau(\sum_j p_{i_j}) \leq 1$.  Since each $\tau(p_i) > 0$, the family is countable, say $(p_i)_{i \in \bbN}$, and the partial sums $s_N = \sum_{i \leq N} p_i$ satisfy $d(s_N, s_M)^3 = \sum_{N < i \leq M} \tau(p_i)$.  They converge in $D_1$ to a projection $p$ with $pp_i = p_i$ for all $i$, and any projection $q$ with $qp_i = p_i$ for all $i$ satisfies $qs_N = s_N$ for all $N$, hence $qp = p$.
\end{proof}

\begin{lemma}\label{lem:masas}
    Every maximal abelian self-adjoint subalgebra of $A$ is the closed linear span of its projections.
\end{lemma}

\begin{proof}
    Let $B$ be such a subalgebra and $h \in B$ self-adjoint.  By Lemma~\ref{lem:borel-calculus} the projections $\Phi_h(\chi_E)$ commute with $B$, so they lie in $B$.  For $\delta > 0$ choose a step function $s$ with $|t - s(t)| \leq \delta$ on $\sigma(h)$; then $\|h - \Phi_h(s)\| \leq \delta$.
\end{proof}

By Lemmas~\ref{lem:orthogonal-sups} and~\ref{lem:masas} and Kaplansky's characterization~\cite[Theorem~2.3]{Kap51}, $A$ is an \AWstar-algebra.

\begin{lemma}\label{lem:full-linearity}
    The function $\tau$ is linear on every abelian \cstar-subalgebra of $A$.
\end{lemma}

\begin{proof}
    Let $a,b \in A$ be commuting self-adjoint elements and let $B$ be a maximal abelian self-adjoint subalgebra containing them.  By Lemma~\ref{lem:masas}, $a$ and $b$ are norm limits of linear combinations of countably many projections in $B$.  These generate a countable Boolean algebra of commuting projections, and the \cstar-algebra $C$ they generate is $C(K)$ for its Stone space $K$, which is compact, metrizable and totally disconnected.  Hence $K$ embeds in the Cantor set, and $C = \text{\cstar}(1,k)$ for the corresponding self-adjoint coordinate function $k$ by the Stone–Weierstrass theorem.  Since $a,b \in C$, Lemma~\ref{lem:tau-properties}(d) gives $\tau(a + b) = \tau(a) + \tau(b)$ and $\tau(\lambda a) = \lambda\tau(a)$, and Lemma~\ref{lem:tau-properties}(b) completes the proof.
\end{proof}

We now have all of the ingredients to prove Theorem~\ref{thm:elementary}.

\begin{proof}[Proof of Theorem~\ref{thm:elementary}]
    By Lemmas~\ref{lem:cstar-structure}--\ref{lem:full-linearity}, $A$ is a unital \AWstar-algebra and $\tau$ is a faithful normalized function on $A$ satisfying (i)–(iii) of the definition of a quasitrace.  Since $A$ is an \AWstar-algebra, $\tau$ extends to $M_2(A)$~\cite[Corollary~II.1.10]{BH82}, so $\tau$ is a quasitrace, and its unit ball $D_1$ is $\dt$-complete.  By~\cite[Lemma~3.9]{Haa14} it is normal, so $(A,\tau)$ is a quasitracial \AWstar-algebra and $\mathfrak{M}$ is its associated structure.  Uniqueness up to isomorphism is clear, since $A$ is recovered as the direct limit of the sorts.
\end{proof}

\begin{lemma}\label{lem:definability}
    The following are definable uniformly in quasitracial \AWstar-algebras.
    \begin{enumerate}
        \item[(a)] Continuous functional calculus on self-adjoint norm balls.
        \item[(b)] The self-adjoint and positive parts of each norm ball, and the positive unit ball.
        \item[(c)] The quasitrace on positive norm balls.
        \item[(d)] The order predicate
        \[
            \ord(x,y) = \dt(\negp{y - x}, 0),
        \]
        whose zero set is $\{(x,y) \ : \ x \leq y\}$.
        \item[(e)] The order-defect predicate
        \[
            \theta(x,y) = \tau(\pos{x - y})
        \]
        on the self-adjoint unit ball.
        \item [(f)]  The set $\Proj(A)$ of projections.  More precisely, for $x \in D_1$ put
        \[
            \varphi(x) = \dt(x, x^*) + \dt(x, x^2);
        \]
        then $\dist(x, \Proj(A)) \leq 5\varphi(x)$.
    \end{enumerate}
\end{lemma}

\begin{proof}
    Continuous functional calculus is obtained by uniform polynomial approximation on each fixed interval.  The maps
    \[
        x \longmapsto \tfrac{1}{2}(x + x^*) \qquad \text{and} \qquad x \longmapsto \pos{\tfrac{1}{2}(x + x^*)}
    \]
    are therefore definable retractions onto the self-adjoint and positive parts.  For $a \geq 0$,
    \[
        \tau(a) = \dt(a^{1/2},0)^3,
    \]
    giving (c), and (d) and (e) follow.

    For (f), let $h$ be the real part of $x$ and $k = x - h$.  We compute that $\dt(x,h) = 2^{-2/3}\dt(x,x^*)$.  Expanding 
    \[
        h^2 - h = (x^2 - x) - xk - kx + k^2 + k
    \]
    and using Fact~\ref{fact:haagerup-metric}(c) with $\|x\|, \|k\| \leq 1$ gives
    \[
        \dt(h^2,h) \leq \dt(x^2,x) + 4\dt(k,0).
    \]
    Let $p = E_{[1/2,\infty)}(h)$ be the spectral projection of $h$ corresponding to $[1/2,\infty)$.  For $t \in [-1,1]$ we have $|t - \chi_{[1/2,\infty)}(t)| \leq 2|t^2 - t|$, so 
    \[
        (h - p)^2 \leq 4(h^2 - h)^2,
    \]
    in the abelian \AWstar-subalgebra generated by $h$, and therefore we have $\dt(h,p) \leq 4^{1/3}\dt(h^2,h)$.  Combining gives $\dist(x,\Proj(A)) \leq 5\varphi(x)$.
\end{proof}

The following lemma is inspired by the axioms for factoriality in the model theory of tracial von Neumann algebras~\cite{FHS13, FHS14}.

\begin{lemma}\label{lem:swap-estimate}
    Let $N$ be a finite \AWstar-factor with normalized quasitrace $\tau$ and let $e \in N$ be a projection. Then there is a unitary $u \in N$ such that
    \[
        \tau([e,u]^* [e,u]) = 2 \min\{\tau(e), 1 - \tau(e)\}.
    \]
\end{lemma}

\begin{proof}
    Assume without loss of generality that $\tau(e) \leq 1/2$. Comparison of projections in finite factors gives a projection $f \leq 1 - e$ with $f \sim e$.  Choose $v$ with $v^*v = e$ and $vv^* = f$.  Then $u = v + v^* + (1 - e - f)$ is a self-adjoint unitary.  A simple calculation gives $[e,u] = v^* - v$, and therefore $[e,u]^* [e,u] = e + f$.  Since equivalent projections have the same quasitrace and $f$ is orthogonal to $e$, we get the desired equality.
\end{proof}

The class of finite \AWstar-factors is also elementary. We give a single extra explicit axiom.  For $x \in D_1$ put
\[
    g(x) = 2\min\{\tre(x), 1 - \tre(x)\} \dminus \sup_{y \in D_1} \tre([x,y]^*[x,y]),
\]
and let $\Tfac$ be $\Tqt$ together with the axiom
\begin{enumerate}
    \item[(F)] $\displaystyle\sup_{x \in D_1}\Bigl(g(x) \dminus \bigl(27\varphi(x)^{3/4} + 80\varphi(x)\bigr)\Bigr) = 0.$
\end{enumerate}

\begin{theorem}\label{thm:elementary-factors}
    The models of $\Tfac$ are exactly the structures associated to finite \AWstar-factors with their canonical quasitraces.  In particular, the class $K_{\mathrm{fac}}$ of quasitracial \AWstar-factors is axiomatized by the explicit recursive theory $\Tfac$.
\end{theorem}

\begin{proof}
    Let $x \in D_1$ and let $p$ be a projection.  The modulus of $\tre$ for $n = 1$ gives a contribution of at most $8\dt(x,p)^{3/4}$ from the first term of $g$.  For $y \in D_1$ we have $\dt([x,y],[p,y]) \leq 2\dt(x,p)$ by Fact~\ref{fact:haagerup-metric}(c), and $\dt(z,0) \leq 2^{2/3}$ when $\|z\| \leq 2$, so (Q7) gives
    \[
        \bigl|\tre([x,y]^*[x,y]) - \tre([p,y]^*[p,y])\bigr| \leq 3 \cdot 2^{4/3} \cdot 2\dt(x,p) \leq 16\dt(x,p).
    \]
    Hence $|g(x) - g(p)| \leq 8\dt(x,p)^{3/4} + 16\dt(x,p)$.

    Let $M$ be a finite \AWstar-factor.  Then $\bM \models \Tqt$ by Proposition~\ref{prop:correctness1}, and $g(p) = 0$ for every projection $p$ by Lemma~\ref{lem:swap-estimate}.  Given $x \in D_1$, Lemma~\ref{lem:definability}(f) gives a projection $p$ with $\dt(x,p) \leq 5\varphi(x)$, and then $g(x) \leq 8 \cdot 5^{3/4}\varphi(x)^{3/4} + 80\varphi(x) \leq 27\varphi(x)^{3/4} + 80\varphi(x)$.

    Conversely, let $\mathfrak{M} \models \Tfac$.  By Theorem~\ref{thm:elementary}, $\mathfrak{M}$ is the structure associated to a quasitracial \AWstar-algebra $(A,\tau)$.  If $p$ were a central projection other than $0$ and $1$, then $\varphi(p) = 0$ and $[p,y] = 0$ for all $y$, so (F) would give $2\min\{\tau(p), 1 - \tau(p)\} = 0$, contradicting faithfulness.  The center of an \AWstar-algebra is generated by its projections, so $A$ is a finite \AWstar-factor, and uniqueness of the normalized quasitrace on finite factors (Fact~\ref{fact:finite-factors}) implies that $\tau$ is its canonical quasitrace.
\end{proof}

In the von Neumann algebra setting, it is a useful fact that tracial ultraproducts of factors are factors.  Nagy claims the analogous result for quasitracial ultraproducts indexed by $\bbN$ \cite{Nag07}.  The following corollary gives the analogous result for arbitrary index sets.

\begin{corollary}\label{cor:factor-ultraproduct}
    A quasitracial ultraproduct of finite \AWstar-factors with their canonical quasitraces is again a finite \AWstar-factor.
\end{corollary}

\begin{proof}
    By Lemma~\ref{lem:ultraproduct}, the quasitracial ultraproduct is the ultraproduct of the associated $\Lqt$-structures.  \L o\'s's theorem then implies it is a model of $\Tfac$ and hence a finite \AWstar-factor by Theorem~\ref{thm:elementary-factors}. 
\end{proof}

\section{Reflection}\label{sec:reflection}

\begin{proposition}\label{prop:reflection}
    Let $M$ be a finite \AWstar-factor and $S \subseteq M$ be countable. Suppose that the positive unit ball of $M$ contains an increasing net $(a_i)_{i \in I}$ with no supremum.  Then there is a separable $\bN \preceq \bM$ containing $S$ and an increasing sequence $(b_n)$ from $\{a_i \ : \ i \in I\}$ with no supremum in $N_{\sa}$.
\end{proposition}

\begin{proof}
    For $m \geq 1$ let
    \[
        U_m = \{y \in M_{\sa} \ : \ \|y\| \leq m,\ a_i \leq y \text{ for all } i \in I\}.
    \]
    For self-adjoint $s \in D_1(M)$ put
    \[
        \alpha(s) = \sup_{i \in I} \ord(a_i, s), \qquad \beta_m(s) = \sup_{y \in U_m} \ord(s, y).
    \]
    These functions have the uniform continuity modulus of $\ord$.  A supremum of the net would lie in the positive unit ball, and $s$ is such a supremum exactly when $\alpha(s) = 0$ and $\beta_m(s) = 0$ for all $m$.

    Fix $F : I \times I \to I$ with $F(i,j) \geq i,j$.  Start with $J_0 = \varnothing$.  By the downward L\"owenheim--Skolem theorem, choose $\bN_0\preceq\bM$ of density character $\aleph_0$ containing $S$, and a countable dense set $S_0\subseteq D_1(N_0)_{\sa}$.  Recursively construct increasing elementary substructures $\bN_n\preceq\bM$, countable $F$-closed sets $J_n\subseteq I$, and countable dense sets $S_n\subseteq D_1(N_n)_{\sa}$.  At stage $n$, for each $s \in S_n$ choose $i(s,n)$ with
    \[
        \ord(a_{i(s,n)}, s) > \alpha(s) - 2^{-n},
    \]
    and, for $m \leq n$, choose $y(s, m, n) \in U_m$ with
    \[
        \ord(s, y(s, m, n)) > \beta_m(s) - 2^{-n}.
    \]
    Enlarge
    \[
        J_n \cup \{ i(s,n) \ : \ s \in S_n\}
    \]
    to a countable $F$-closed set $J_{n + 1} \subseteq I$.  Using the
    downward L\"owenheim–Skolem theorem, choose
    \[
        N_n \preceq N_{n + 1} \preceq M
    \]
    of density character $\aleph_0$ so that $N_{n + 1}$ contains all of the
    chosen elements $y(s, m, n)$ and, in addition,
    \[
        \{a_j \ : \ j \in J_{n + 1}\} \subseteq N_{n + 1}.
    \]
    Finally choose a countable dense set
    \[
        S_{n + 1} \subseteq D_1(N_{n + 1})_{\sa}.
    \]

    Let $\bN$ be the completion of $\bigcup_n \bN_n$ and $J = \bigcup_n J_n$.  The construction and uniform continuity give, for self-adjoint $s \in D_1(N)$,
    \[
        \sup_{j \in J} \ord(a_j, s) = \alpha(s),
    \]
    and the chosen upper bounds from $N \cap U_m$ compute $\beta_m(s)$.

    The countable directed set $J$ has an increasing cofinal sequence $(k_n)$.  Put $b_n = a_{k_n}$.  If $s \in N$ were the supremum of $(b_n)$, then $s$ would dominate every $a_j$ with $j \in J$, hence $\alpha(s) = 0$.  Thus $s$ is an upper bound for the original net.  Since the original net has no supremum, $\beta_m(s) > 0$ for some $m$, so some chosen $y \in N \cap U_m$ satisfies $\ord(s, y) > 0$. Then $y$ is an upper bound for $(b_n)$ in $N$ with $s \not\leq y$, contradicting the leastness of $s$ in $N_{\sa}$.
\end{proof}

\begin{theorem}\label{thm:elementary-reflection}
    Let $T = \Th(M)$ be the complete theory of a finite \AWstar-factor.  Then the following are equivalent:
    \begin{itemize}
        \item[(i)]  every model of $T$ is monotone complete,
        \item[(ii)]   every $\dt$-separable model of $T$ is monotone sequentially complete.
    \end{itemize}
\end{theorem}

\begin{proof}
    (i) $\implies$ (ii) is immediate.

    For (ii) $\implies$ (i) assume $P \models T$ is not monotone complete.  By Theorem~\ref{thm:elementary-factors}, $P$ is a finite \AWstar-factor.  By translating and scaling, we can ensure that the failure of monotone completeness occurs on the positive unit ball.  Proposition~\ref{prop:reflection} applied to $P$ yields separable $N \preceq P$ which is not monotone sequentially complete.  This contradicts (ii).
\end{proof}

\begin{corollary}\label{cor:separable-reduction}
    Every finite \AWstar-factor is monotone complete if and only if every $\dt$-separable \AWstar-factor of type $\mathrm{II}_1$ is monotone sequentially complete.
\end{corollary}

\begin{proof}
    One direction is immediate.  For the other, suppose $M$ is a finite \AWstar-factor which is not monotone complete.  As in the proof of Theorem~\ref{thm:elementary-reflection}, Proposition~\ref{prop:reflection} gives a $\dt$-separable $N \preceq M$ which is not monotone sequentially complete, and $N$ is a finite \AWstar-factor by Theorem~\ref{thm:elementary-factors}.  Factors of type $\mathrm{I}_n$ are finite dimensional, hence monotone complete, so by Fact~\ref{fact:finite-factors} $N$ is of type $\mathrm{II}_1$.
\end{proof}

\section{Arithmetic encoding and absoluteness}\label{sec:absolute}

We have seen that $\Tqt$ and $\Tfac$ are explicit computable theories.  We now recall the completeness theorem of Ben Yaacov and Pedersen for continuous logic.  Let $\bbD$ denote the dyadic rationals.

\begin{fact}[{\cite[Corollary~9.8]{BYP10}}]\label{fact:completeness-theorem}
    Let $L$ be a recursive language, $T$ a recursive $L$-theory, and $\sigma$ an $L$-sentence in the basic syntax of~\cite{BYP10}.  Then
    \[
        \sup\{\sigma^M \ : \ M \models T\}
         = 
        \inf\{q \in \bbD_{>0} \ : \ T \vdash \sigma \dminus q\}.
    \]
\end{fact}

The formulas used below employ the recursively approximable derived connectives fixed in Section~\ref{sec:language}.  To make the arithmetization literal, let $\varphi$ be any such displayed sentence and let $\varphi^{[m]}$ be its fixed basic-syntax approximants.  For a positive dyadic $q$, write
\[
    \operatorname{Prv}_T(\varphi < q)
\]
for the arithmetic assertion that there are $m \in \bbN$ and $s \in \bbD_{\geq 0}$ such that
\[
    s + 2^{-m} < q \qquad \text{and} \qquad T \vdash \varphi^{[m]} \dminus s.
\]
Since $T$ is recursive and the approximation scheme is recursive, this is a $\Sigma^0_1$ relation in the codes of $T,\varphi,q$.  Fact~\ref{fact:completeness-theorem} and the uniform approximation give
\begin{equation}\label{eq:effective-completeness}
    \operatorname{Prv}_T(\varphi<q) \quad \Longleftrightarrow \quad \sup\{\varphi^M \ : \ M \models T\}<q.
\end{equation}
Indeed, the forward implication is soundness plus the error $2^{-m}$.  Conversely, choose $m$ so that $\sup_M\varphi^M + 2\cdot 2^{-m} < q$.  Then $\sup_M(\varphi^{[m]})^M < q-2^{-m}$, so Fact~\ref{fact:completeness-theorem} gives a dyadic $s < q-2^{-m}$ with $T \vdash \varphi^{[m]} \dminus s$.

We first give an explicit sentence encoding Kaplansky's problem.  Let $\delta$ be the $\Lqt$-sentence
\[
    \delta = \frac{1}{2}\sup_{x,y \in D_1} \bigl|\tre(x^*x + y^*y) - \tre(x^*x) - \tre(y^*y)\bigr|.
\]

\begin{proposition}\label{prop:Kap-encoding}
    Let $M$ be a finite \AWstar-factor.  Then $\delta^{\bM} = 0$ if and only if $M$ is a von Neumann algebra.
\end{proposition}

\begin{proof}
    If $\delta^{\bM} = 0$, then $\tau_M$ is additive on the positive part of the unit ball, and hence on $M_{+}$ by scaling.  It follows that $\tau_M$ is additive on $M_{\sa}$ by writing $a = (a + \|a\|) - \|a\|$.  It follows by (Q5) that $\tau_M$ is linear.  Thus $\tau_M$ is a trace and therefore $M$ is a von Neumann algebra by Fact~\ref{fact:finite-factors}.  Conversely, if $M$ is a von Neumann algebra, then $\tau_M$ is a trace and $\delta^{\bM} = 0$.
\end{proof}

For the monotone completeness conjecture we need a uniform characterization by monotone chains.  The next definition is inspired by stability theory: a monotone half-graph is an instance of the order property for the predicate $\theta$~\cite{BYU10}.

\begin{definition}\label{def:monotone-chains}
    Let $(A,\tau)$ be a quasitracial \AWstar-algebra.  A \textbf{monotone half-graph of gap $\eta > 0$} is a sequence $(c_n)$ of positive contractions such that $c_n \leq c_m$ and $\theta(c_m, c_n) \geq \eta$ whenever $n < m$.  
    
    For $\eta > 0$ and $\varepsilon \geq 0$, an \textbf{$(\eta,\varepsilon)$-chain of length $k$} is a tuple $c_1, \ldots, c_k$ of positive contractions such that $\theta(c_i, c_j) \leq \varepsilon$ and $\theta(c_j, c_i) \geq \eta$ whenever $i < j$.
\end{definition}

\begin{proposition}\label{prop:half-graph}
    If a quasitracial \AWstar-algebra $(A,\tau)$ admits no monotone half-graph of positive gap, then every increasing net in the positive unit ball of $A$ converges in $\dt$ to its supremum.  In particular $A$ is monotone complete.
\end{proposition}

\begin{proof}
    Let $(a_i)$ be an increasing net of positive contractions of $A$ that is not $\dt$-Cauchy.  Then there exists $\delta > 0$ such that for all $i$, there are $j,k \geq i$ with $\dt(a_j, a_k) \geq \delta$.  By the triangle inequality, we may assume without loss of generality that $\dt(a_j, a_i) \geq \frac{\delta}{2}$.  Since $0 \leq a_j - a_i \leq 1$, Lemma~\ref{lem:abelian-comparison} gives $\eta : = (\frac{\delta}{2})^3 \leq \tau(a_j - a_i)$.  Starting from any $i_0$, choose recursively $i_{n + 1} \geq i_n$ with $\tau(a_{i_{n + 1}} - a_{i_n}) \geq \eta$.  For $n < m$, we have $a_{i_m} - a_{i_n} \geq a_{i_{n + 1}} - a_{i_n} \geq 0$.  Since $\tau$ is order-preserving, we have $\theta(a_{i_m}, a_{i_n}) = \tau(a_{i_m} - a_{i_n}) \geq \eta$.  Then $(a_{i_n})$ defines a monotone half-graph in $A$ of gap $\eta$.

    Thus, if $A$ admits no monotone half-graphs, then all increasing nets are $\dt$-Cauchy and hence converge.  Continuity of $\ord$ makes the limit agree with the supremum.
\end{proof}

\begin{lemma}\label{lem:saturated-half-graph}
    No monotone sequentially complete $\aleph_1$-saturated quasitracial \AWstar-algebra admits a monotone half-graph of positive gap.
\end{lemma}

\begin{proof}
    Suppose otherwise.  Let $A$ be a monotone sequentially complete $\aleph_1$-saturated quasitracial \AWstar-algebra and $(c_n)$ be a monotone half-graph in $A$ with gap $\eta$ and supremum $s$.  
    
    Assume toward an intermediate contradiction that $c_n$ does not converge to $s$ in $\dt$.  Then there exist $\delta > 0$ and infinitely many $m$ with $\dt(c_m, s) \geq \delta$.  The countable type
    \[
        c_n \leq x \quad (n \in \bbN), \qquad x \leq s, \qquad \dt(x,s) \geq \delta
    \]
    is finitely satisfied by $c_m$ for large $m$, so saturation gives an upper bound $x \leq s$ with $x \not= s$, which is a contradiction.
    
    Hence $(c_n)$ is $\dt$-Cauchy, and for $n < m$ Lemma~\ref{lem:abelian-comparison} gives $\theta(c_m,c_n) = \tau(c_m - c_n) \leq \dt(c_m, c_n)^{3/2}$, which tends to $0$.  This contradicts $\theta(c_m, c_n) \geq \eta$.
\end{proof}

\begin{proposition}\label{prop:uniform-chains}
    Every finite \AWstar-factor is monotone complete if and only if, for every $\eta > 0$, there are $\varepsilon > 0$ and $k$ such that no finite \AWstar-factor contains an $(\eta,\varepsilon)$-chain of length $k$.
\end{proposition}

\begin{proof}
    Suppose the uniform bounds hold, and let $M$ be a finite \AWstar-factor which is not monotone complete.  Proposition~\ref{prop:half-graph} gives a monotone half-graph $(c_n)$ of some gap $\eta > 0$, and for every $\varepsilon$ and $k$ the tuple $c_1, \ldots, c_k$ is an $(\eta,\varepsilon)$-chain, a contradiction.

    Conversely, suppose every finite \AWstar-factor is monotone complete and the bounds fail at some $\eta > 0$.  For each $j$ choose a finite \AWstar-factor $M_j$ containing an $(\eta,1/j)$-chain of length $j$, and let $M = \prod_{\cU} M_j$ for a nonprincipal ultrafilter $\cU$ on $\bbN$.  By Corollary~\ref{cor:factor-ultraproduct}, $M$ is a finite \AWstar-factor, and it is $\aleph_1$-saturated~\cite{BBHU08}.  Taking the first $k$ entries of the chains for $j \geq k$ and applying \L o\'s's theorem gives $(\eta,0)$-chains of every length in $M$.  So the countable type in $(x_n)$ asserting that each $x_n$ is a positive contraction with $\theta(x_n,x_m) = 0$ and $\theta(x_m,x_n) \geq \eta$ for $n < m$ is finitely satisfiable, and saturation realizes it.  Since $\tau$ is faithful, the realization is a monotone half-graph of gap $\eta$.  But $M$ is monotone complete by hypothesis, so this contradicts Lemma~\ref{lem:saturated-half-graph}.
\end{proof}

We now express these chain conditions by definable sentences.  For dyadic $\eta \in (0,1]$, dyadic $\varepsilon \in [0,1]$ and $k \geq 2$, let
\[
    \psi_{\eta,\varepsilon,k} = \inf_{x_1, \ldots, x_k \in D_1} \max_{i < i' \leq k} \max\bigl\{\theta(c_i,c_{i'}) \dminus \varepsilon,\, \eta \dminus \theta(c_{i'},c_i) \bigr\},
\]
where $c_i = x_i^*x_i$.  By Lemma~\ref{lem:definability} and the conventions of Section~\ref{sec:language}, the family $\psi_{\eta,\varepsilon,k}$ has recursive basic-syntax approximants uniformly in the dyadic parameters and in $k$.

\begin{lemma}\label{lem:chain-attain}
    No finite \AWstar-factor contains an $(\eta,\varepsilon)$-chain of length $k$ if and only if there is a rational $r > 0$ with $\Tfac \models \psi_{\eta,\varepsilon,k} \geq r$.
\end{lemma}

\begin{proof}
    A chain witnesses $\psi_{\eta,\varepsilon,k} = 0$.  Conversely, if models $\bM_j$ of $\Tfac$ satisfy $\psi^{\bM_j}_{\eta,\varepsilon,k} \to 0$, then a nonprincipal ultraproduct is a model of $\Tfac$ with $\psi_{\eta,\varepsilon,k} = 0$.  It is the structure of a finite \AWstar-factor by Theorem~\ref{thm:elementary-factors}, and since it is $\aleph_1$-saturated~\cite{BBHU08} the infimum is attained, which produces a chain.
\end{proof}

For $n \in \bbN$ let $\pi_K(n)$ be the $\Sigma^0_1$ sentence
\[
    \operatorname{Prv}_{\Tfac}(\delta < 2^{-n}).
\]
For $n \in \bbN$ let $\pi_{\mathrm{MC}}(n)$ be the $\Sigma^0_1$ sentence asserting that there are dyadic $\varepsilon,r > 0$ with $\varepsilon \leq 1$, $r \leq \tfrac{1}{2}$, and an integer $k \geq 2$ such that
\[
    \operatorname{Prv}_{\Tfac}
    \bigl((2r\dminus\psi_{2^{-n},\varepsilon,k}) < r\bigr).
\]
Put
\[
    \sigma_{\mathrm{MC}}:\ \forall n \in \bbN\ \pi_{\mathrm{MC}}(n)
\]
and
\[
    \sigma_K:\ \forall n \in \bbN\ \pi_K(n).
\]
Both are $\Pi^0_2$ sentences of arithmetic.

For a unital \cstar-algebra $A$, let $\mu(A)$ be the smallest $n \geq 1$ for which $M_n(A)$ is properly infinite, and let $\nu(A)$ be the smallest $n \geq 1$ for which there exist $x_1, \ldots, x_n \in A$ with $\sum_{i = 1}^n x_i^* x_i = 1$ and $\| \sum_{i = 1}^n x_i x_i^* \| \leq 1/2$, each being $\infty$ if no such $n$ exists; thus $\nu(A) < \infty$ exactly when $A$ has no tracial state.  Milh{\o}j and R{\o}rdam \cite[Remark 4.7]{MR23} observe that a counterexample to Kaplansky's problem can be obtained as an ultraproduct unless there is a universal relation $\mu(A) \leq f(\nu(A))$ for all traceless unital \cstar-algebras $A$, where $f : \bbN \to \bbN$ is fixed.  Together with the uniqueness of the quasitrace on finite \AWstar-factors, this shows that Kaplansky's problem has a positive answer if and only if such a relation exists.  Combining this with standard continuous logic facts (in the language of unital \cstar-algebras) gives a $\Pi^{0}_{2}$ sentence for Kaplansky's conjecture.  The next theorem supplies an alternative version in the language of quasitracial \AWstar-algebras.  We know of no such precedent for the monotone completeness conjecture half of the following theorem.

\begin{theorem}\label{thm:absoluteness}
    There are $\Pi^{0}_{2}$ sentences $\sigma_{\mathrm{MC}}$ and $\sigma_K$ of arithmetic such that ZFC proves
    \[
        \sigma_{\mathrm{MC}} \quad \Longleftrightarrow \quad \text{every finite \AWstar-factor is monotone complete},
    \]
    and
    \[
        \sigma_K \quad \Longleftrightarrow\quad \text{every type $\mathrm{II}_1$ \AWstar-factor is a von Neumann algebra}.
    \]
    Consequently each assertion has the same truth value in every $\omega$-model of ZFC, in particular in $V$, in $L$, in every inner model, and in every set-generic extension of $V$.
\end{theorem}

\begin{proof}
    By~\eqref{eq:effective-completeness}, $\sigma_K$ is equivalent to
    \[
        \sup\{\delta^{\bM}:\bM\models\Tfac\} = 0.
    \]
    By Theorem~\ref{thm:elementary-factors} and Proposition~\ref{prop:Kap-encoding}, this holds exactly when every finite \AWstar-factor is a von Neumann algebra; factors of type $\mathrm{I}_n$ are matrix algebras.

    For $\sigma_{\mathrm{MC}}$, it suffices to consider gaps $\eta = 2^{-n}$: if the chain bound holds for a smaller dyadic gap, then it holds for every larger gap.  The tolerance $\varepsilon$ may likewise be decreased to a positive dyadic number.

    Suppose first that no finite factor contains a $(2^{-n},\varepsilon)$-chain of length $k$.  By Lemma~\ref{lem:chain-attain} there is $r'>0$ such that
    \[
        \Tfac\models\psi_{2^{-n},\varepsilon,k} \geq r'.
    \]
    Choose a dyadic $r$ with $0 < 2r < r'$ and $r \leq \tfrac{1}{2}$.  Then
    \[
        \Tfac\models(2r\dminus\psi_{2^{-n},\varepsilon,k}) = 0,
    \]
    so~\eqref{eq:effective-completeness} gives
    \[
        \operatorname{Prv}_{\Tfac}
        \bigl((2r\dminus\psi_{2^{-n},\varepsilon,k}) < r\bigr).
    \]
    Thus $\pi_{\mathrm{MC}}(n)$ holds.

    Conversely, if $\pi_{\mathrm{MC}}(n)$ holds, then for some dyadic $\varepsilon,r > 0$ and $k \geq 2$, equation~\eqref{eq:effective-completeness} gives
    \[
        \Tfac\models(2r\dminus\psi_{2^{-n},\varepsilon,k}) < r.
    \]
    Hence
    \[
        \Tfac\models\psi_{2^{-n},\varepsilon,k} > r > 0,
    \]
    and Lemma~\ref{lem:chain-attain} shows that no finite factor contains a $(2^{-n},\varepsilon)$-chain of length $k$.  Therefore $\sigma_{\mathrm{MC}}$ is equivalent to the chain condition in Proposition~\ref{prop:uniform-chains}, proving the first equivalence.

    These equivalences are theorems of ZFC, so they hold in every model of ZFC.  An $\omega$-model of ZFC has the standard natural numbers and therefore agrees with $V$ on the arithmetic sentences $\sigma_{\mathrm{MC}}$ and $\sigma_K$.
\end{proof}

By \cite{BH82} and \cite[Section~2]{MR23}, a positive answer to Kaplansky's problem is equivalent to the statement that every unital stably finite \cstar-algebra admits a tracial state.  Blackadar and Farah show that this statement is absolute between transitive models of ZFC containing all countable ordinals~\cite{BF26}.  Theorem~\ref{thm:absoluteness} extends this absoluteness to all $\omega$-models of ZFC.

\subsection{Finite \texorpdfstring{\AWstar}{AW*}-algebras}

The model theory above only applies to finite \AWstar-algebras carrying a faithful normal scalar quasitrace, and a finite \AWstar-algebra whose center admits no faithful normal state carries none.  We pass from factors to arbitrary finite \AWstar-algebras using Boolean-valued models as in~\cite{Bell05}.  Let $A$ be a unital \AWstar-algebra with center $Z$ and let $\mathbb{B} = \Proj(Z)$.  By~\cite{Oz86, AH26}, there is a $\mathbb{B}$-name $\mathcal{M}$ such that
\[
    \bv{\mathcal{M} \text{ is an \AWstar-factor}} = 1,
\]
$A$ is the bounded part of $\mathcal{M}$, and
\[
    A \text{ is monotone complete} \iff \bv{\mathcal{M} \text{ is monotone complete}} = 1.
\]

\begin{lemma}\label{lem:finite-BV}
    If $A$ is finite, then
    \[
        \bv{\mathcal{M} \text{ is a finite \AWstar-factor}} = 1.
    \]
\end{lemma}

\begin{proof}
    Let
    \[
        b = \bv{\exists x \in \mathcal{M}\, (x^*x = 1 \wedge xx^* \not= 1)}.
    \]
    By the maximum principle there is a name $x''$ with $\bv{x''^*x'' = 1 \wedge x''x''^* \not= 1} = b$.  Mixing $x''$ on $b$ with $1$ on $\neg b$ gives a name $x'$ with $\bv{x'^*x' = 1} = 1$ and $\bv{x'x'^* \not= 1} \geq b$.  Since $x'$ is internally an isometry, it has norm $1$, so it lies in the bounded part of $\mathcal{M}$ and corresponds to an element $v \in A$ with $v^*v = 1$.  Since $A$ is finite, $vv^* = 1$, so $\bv{x'x'^* = 1} = 1$ and therefore $b = 0$.
\end{proof}

Arithmetic sentences are absolute between $V$ and $V^{\mathbb{B}}$: the internal natural numbers of $V^{\mathbb{B}}$ are given by $\check\omega$, and bounded formulas with check-name parameters have Boolean value $1$ exactly when they hold in $V$~\cite{Bell05}.  In particular $\bv{\sigma_{\mathrm{MC}}} = 1$ if and only if $\sigma_{\mathrm{MC}}$ holds.

\begin{theorem}\label{thm:finite-algebras}
    ZFC proves that the following are equivalent.
    \begin{enumerate}
        \item[(a)] Every finite \AWstar-algebra is monotone complete.
        \item[(b)] Every finite \AWstar-algebra is monotone sequentially complete.
        \item[(c)] Every finite \AWstar-factor is monotone complete.
        \item[(d)] Every $\dt$-separable \AWstar-factor of type $\mathrm{II}_1$ is monotone sequentially complete.
        \item[(e)] $\sigma_{\mathrm{MC}}$.
    \end{enumerate}
    In particular these assertions have the same truth value in every $\omega$-model of ZFC.
\end{theorem}

\begin{proof}
    The implications (a) $\implies$ (b) $\implies$ (d) are immediate.  Corollary~\ref{cor:separable-reduction} gives (d) $\implies$ (c), and Theorem~\ref{thm:absoluteness} gives (c) $\iff$ (e).

    It remains to prove (c) $\implies$ (a).  Let $A$ be a finite \AWstar-algebra and let $\mathcal{M}$ be as above.  By Theorem~\ref{thm:absoluteness}, (c) implies $\sigma_{\mathrm{MC}}$, so $\bv{\sigma_{\mathrm{MC}}} = 1$ by absoluteness of arithmetic sentences.  Since ZFC proves that $\sigma_{\mathrm{MC}}$ implies (c), and every axiom of ZFC has Boolean value $1$, we get
    \[
        \bv{\text{every finite \AWstar-factor is monotone complete}} = 1.
    \]
    Lemma~\ref{lem:finite-BV} gives $\bv{\mathcal{M} \text{ is monotone complete}} = 1$, so $A$ is monotone complete.

    The final assertion follows from (e), since $\sigma_{\mathrm{MC}}$ is arithmetic.
\end{proof}

\section{Real \texorpdfstring{\AWstar}{AW*}-algebras}\label{sec:real}

Until this section all \AWstar-algebras have been complex.  A real \cstar-algebra is a real Banach $*$-algebra satisfying the real \cstar axioms, and a real \AWstar-algebra is a real \cstar-algebra which is a Baer $*$-ring.  We call $R$ a real \AWstar-factor when $Z(R) = \bbR1$, and write $R_{\bbC} = R + iR$ for its complexification.

\subsection{Semi-abelian projections, complexification and order}

\begin{definition}
    A real Baer $*$-ring $B$ is \emph{semi-abelian} if every projection of $B$ is central.  A projection $e \in R$ is semi-abelian if $eRe$ is semi-abelian.  This is the notion called \emph{abelian} in Berberian and Handelman~\cite{Ber72,Han78}.  We reserve \emph{commutative} for the algebraic condition $xy = yx$.
\end{definition}

The distinction is invisible over the complex numbers but essential over the reals: $\mathbb{H}$ has only the projections $0,1$ and is therefore semi-abelian, but it is not commutative.

\begin{lemma}[Fixed-point algebras]\label{lem:fixed}
    Let $B$ be a real \cstar-algebra, let $\alpha$ be a $*$-automorphism, and let $B^{\alpha} = \{x \in B \ : \ \alpha(x) = x\}$.  Then:
    \begin{enumerate}
        \item[(i)] if $B$ is a real \AWstar-algebra, then so is $B^{\alpha}$;
        \item[(ii)] if $B$ is a normal real \AWstar-algebra, then so is $B^{\alpha}$;
        \item[(iii)] if $B$ is monotone complete, then so is $B^{\alpha}$;
        \item[(iv)] if $B$ is monotone sequentially complete, then so is $B^{\alpha}$.
    \end{enumerate}
\end{lemma}

\begin{proof}
    If $S \subseteq B^{\alpha}$ and $\operatorname{ann}_r^B(S) = pB$, then $\alpha$ preserves this right ideal, so $\alpha(p)B = pB$ and therefore $\alpha(p) = p$.  Hence $\operatorname{ann}_r^{B^{\alpha}}(S) = pB^{\alpha}$, proving (i).  For (ii), the join in $\Proj(B)$ of an upward-directed family from $\Proj(B^{\alpha})$ is fixed by $\alpha$ and is also the projection join in $B^{\alpha}$; normality identifies it with the order supremum.  The same uniqueness argument for suprema in the self-adjoint part proves (iii), and its sequential version proves (iv).
\end{proof}

Put
\[
    J = \begin{pmatrix} 0 & -1 \\ 1 & 0 \end{pmatrix} \in M_2(R).
\]
The fixed-point algebra of $\operatorname{Ad}J$ is
\[
    C_J = \left\{\begin{pmatrix} a & -b \\ b & a \end{pmatrix} \ : \ a,b \in R\right\}.
\]
Since $J$ is central in $C_J$, $J^2 = -1$ and $J^* = -J$, the rule $i\cdot x = Jx$ makes $C_J$ a complex \cstar-algebra, and
\[
    a + ib \longmapsto \begin{pmatrix} a & -b \\ b & a \end{pmatrix}
\]
is a complex $*$-isomorphism $R_{\bbC} \cong C_J$.

\begin{proposition}\label{prop:m2-complex}
    For every real \cstar-algebra $R$, $M_2(R)$ is a real \AWstar-algebra if and only if $R_{\bbC}$ is a complex \AWstar-algebra.
\end{proposition}

\begin{proof}
    If $M_2(R)$ is real \AWstar, Lemma~\ref{lem:fixed}(i) makes $C_J$ a real \AWstar-algebra.  Right annihilators depend only on the ring multiplication, so the same projection generators work after declaring $i\cdot x = Jx$.  Thus $C_J \cong R_{\bbC}$ is complex \AWstar.

    Conversely, suppose $R_{\bbC}$ is complex \AWstar.  Then $M_2(R_{\bbC})$ is complex \AWstar: on the finite central summand this is Berberian's matrix theorem~\cite{Ber58}, while a properly infinite \AWstar-algebra is $*$-isomorphic to its $2 \times 2$ matrix algebra~\cite{Ber72}.  Entrywise canonical conjugation is a real $*$-automorphism of $M_2(R_{\bbC})$ with fixed-point algebra $M_2(R)$, so Lemma~\ref{lem:fixed}(i) applies.
\end{proof}

\begin{theorem}[Handelman complexification theorem]\label{thm:handelman}
    Let $R$ be a real \AWstar-algebra with no nonzero central semi-abelian projection.  Then $R_{\bbC}$ is a complex \AWstar-algebra.
\end{theorem}

\begin{proof}
    Handelman's Corollary~3.12 applies exactly to real \AWstar-algebras with no nonzero central semi-abelian projection, in the terminology above, and gives that every matrix ring over $R$ is a real \AWstar-algebra~\cite[Corollary~3.12]{Han78}.  In particular $M_2(R)$ is real \AWstar, so Proposition~\ref{prop:m2-complex} gives the conclusion.
\end{proof}

Whenever Theorem~\ref{thm:handelman} applies, Albeverio, Ayupov and Abduvaitov's type-comparison theorem~\cite{AAA04} gives coincidence of the refined types $\mathrm{I}_{\mathrm{fin}},\mathrm{I}_{\infty},\mathrm{II}_1,\mathrm{II}_{\infty},\mathrm{III}$.

\begin{lemma}\label{lem:semi-selfadjoint}
    If $B$ is a semi-abelian real \AWstar-algebra, then
    \[
        B_{\sa} = Z(B)_{\sa}.
    \]
    Consequently $B$ is finite, monotone complete and normal.
\end{lemma}

\begin{proof}
    Let $a = a^* \in B$.  By the spectral theorem for \AWstar-algebras, $a$ is uniformly approximated by finite real linear combinations of its spectral projections (equivalently, by step-function approximants formed from the support projections of $(a - t)_{+}$).  Since $B$ is semi-abelian, those projections are central.  Hence $a$ is a norm limit of central elements and is central.  The reverse inclusion is tautological.
    
    The center's self-adjoint part is a commutative real \AWstar-algebra, hence is $C(X,\bbR)$ for a Stonean space $X$~\cite{Li03}; it is monotone complete and normal.  Since the order of a real \cstar-algebra lives entirely in its self-adjoint part, the same is true of $B$.
    
    Finally, if $v^*v = 1$, then $vv^*$ is a projection and hence central.  The central projections $vv^*$ and $1$ are equivalent.  Equivalent central projections are equal: if $w^*w = p$, $ww^* = q$ and $p,q$ are central, then $w = pw$ gives $q = ww^* = pww^* = pq \leq p$, and symmetrically $p \leq q$.  Thus $vv^* = 1$ and $B$ is finite.
\end{proof}

\begin{lemma}\label{lem:semi-factor}
    A semi-abelian real \AWstar-factor is $*$-isomorphic to either $\bbR$ or $\mathbb{H}$.
\end{lemma}

\begin{proof}
    By Lemma~\ref{lem:semi-selfadjoint}, $R_{\sa} = \bbR 1$.  Hence for every $x \in R$ there are $\lambda,\mu \geq 0$ such that
    \[
        x^*x = \lambda1,\qquad xx^* = \mu 1.
    \]
    If $x \not= 0$, then $\lambda,\mu > 0$, and the \cstar-identity gives
    $\lambda = \|x\|^2 = \mu$.  Thus $x^{-1} = \lambda^{-1}x^*$, and $R$ is a real Banach division algebra.  The real Gelfand--Mazur theorem gives $R \cong \bbR,\bbC$ or $\mathbb{H}$; the factor condition $Z(R) = \bbR1$ excludes $\bbC$.
\end{proof}

\begin{corollary}\label{cor:factor-complex}
    The complexification of every real \AWstar-factor is a complex \AWstar-factor of the same refined type.  Consequently every real \AWstar-factor is normal.
\end{corollary}

\begin{proof}
    If the factor is semi-abelian, Lemma~\ref{lem:semi-factor} reduces to $\bbR_{\bbC} = \bbC$ and $\mathbb{H}_{\bbC} \cong M_2(\bbC)$; these are complex \AWstar-factors of finite type I and the real factors are finite-dimensional real von Neumann algebras.
    
    Otherwise the factor has no nonzero central semi-abelian projection, so Theorem~\ref{thm:handelman} makes $R_{\bbC}$ complex \AWstar.  If $a + ib$ is central in $R_{\bbC}$, then $[a,r] + i[b,r] = 0$ for every $r\in R$, hence $a,b\in Z(R) = \bbR1$; thus $R_{\bbC}$ is a factor.  Type coincidence follows from~\cite{AAA04}.  Sait\^o--Wright normality of complex \AWstar-factors~\cite{SW91}, followed by Lemma~\ref{lem:fixed}(ii), gives normality of $R$.
\end{proof}

\begin{theorem}\label{thm:faithful-state}
    Every semifinite real \AWstar-factor admitting a faithful state is a real \wstar-factor.
\end{theorem}

\begin{proof}
    The semi-abelian factor case is finite-dimensional by Lemma~\ref{lem:semi-factor}.  Assume therefore that $R$ is not semi-abelian, so $R_{\bbC}$ is a complex \AWstar-factor of the same refined type by Corollary~\ref{cor:factor-complex}.
    
    Let $\varphi : R \to \bbR$ be a faithful state and put
    \[
     \varphi_s(x) = \frac{1}{2}(\varphi(x) + \varphi(x^*)).
    \]
    Then $\varphi_s$ is positive and faithful, agrees with $\varphi$ on $R_{+}$, and vanishes on skew-adjoint elements.  Moreover
    \[
     |\varphi_s(x)| = |\varphi(\rea x)| \leq \|\rea x\| \leq \|x\|,
     \qquad \varphi_s(1) = 1,
    \]
    so $\varphi_s$ is a state.  Define
    \[
     \widetilde\varphi(a + ib) = \varphi_s(a) + i\varphi_s(b).
    \]
    For $z = a + ib$,
    \[
     z^*z = a^*a + b^*b + i(a^*b - b^*a),
    \]
    and the bracketed term is skew-adjoint, so $\widetilde\varphi(z^*z) = \varphi_s(a^*a + b^*b) \geq 0$.  Faithfulness follows in the same way.  Thus $R_{\bbC}$ has a faithful state.  Type I \AWstar-factors are von Neumann factors, while Wright's results apply in type $\mathrm{II}_1$~\cite{Wri76} and Christensen and Pedersen's results apply in $\mathrm{II}_{\infty}$~\cite{CP84}; hence $R_{\bbC}$ is a complex \wstar-factor.
    
    Canonical conjugation on $R_{\bbC}$ preserves increasing suprema and therefore is ultraweakly continuous: if $\omega$ is normal, then $x\mapsto\overline{\omega(\gamma(x))}$ is again normal, first for positive $\omega$ and then by linear decomposition.  Its fixed-point algebra $R$ is therefore ultraweakly closed in $R_{\bbC}$, and hence is a real \wstar-algebra.
\end{proof}

\begin{corollary}\label{cor:finite-faithful-state}
    Every finite real \AWstar-factor admitting a faithful state is a real \wstar-factor and admits a faithful normalized real trace.
\end{corollary}

\begin{proof}
    The first assertion is Theorem~\ref{thm:faithful-state}.  Its complexification is a finite complex \wstar-factor by Corollary~\ref{cor:factor-complex}; the unique normalized trace is invariant under canonical conjugation and hence restricts to a faithful normalized real trace.
\end{proof}

\begin{theorem}\label{thm:real-proper-infinite}
    Every properly infinite real \AWstar-algebra is monotone sequentially complete.
\end{theorem}

\begin{proof}
    Let $R$ be properly infinite.  The halving theorem gives orthogonal projections $p_1,p_2$ with $p_1 + p_2 = 1$ and $p_j \sim 1$~\cite{Ber72}.  If $0 \not= z \in Z(R)$ were semi-abelian, central compression would give two orthogonal projections $zp_1,zp_2$ in $zR$, each equivalent to the central unit $z$.  Equivalent central projections are equal, contradicting orthogonality.  Thus $R$ has no nonzero central semi-abelian projection.
    
    Theorem~\ref{thm:handelman} makes $R_{\bbC}$ complex \AWstar, and the same halving witnesses show it is properly infinite.  Christensen--Pedersen therefore makes $R_{\bbC}$ monotone sequentially complete~\cite{CP84}.  A bounded increasing sequence in $R_{\sa}$ has a supremum in $(R_{\bbC})_{\sa}$; canonical conjugation fixes the sequence and hence its supremum, which therefore lies in $R_{\sa}$.
\end{proof}

\begin{lemma}\label{lem:semi-split}
    For every real \AWstar-algebra $R$ there is a largest central projection $z$ such that $zR$ is semi-abelian.  The complementary algebra $(1 - z)R$ has no nonzero central semi-abelian projection.
\end{lemma}

\begin{proof}
    Choose a maximal orthogonal family $(z_i)$ of nonzero central semi-abelian projections and let $z = \bigvee_i z_i$.  If $p \leq z$ is a projection, then $z_i p$ is a projection in the semi-abelian algebra $z_iR$, hence central in $z_iR$ and therefore central in $R$.  Since $p = \bigvee_i z_i p$, it follows that $p$ is central.  Thus $zR$ is semi-abelian.  Maximality gives the complementary assertion, and also shows that $z$ is the largest such central projection.
\end{proof}

\begin{theorem}\label{thm:real-normal}
    Every real \AWstar-algebra is normal.
\end{theorem}

\begin{proof}
    Split $R = zR\oplus(1 - z)R$ as in Lemma~\ref{lem:semi-split}.  The semi-abelian summand is normal by Lemma~\ref{lem:semi-selfadjoint}.  The complementary summand has \AWstar-complexification by Theorem~\ref{thm:handelman}; the complex normality theorem~\cite{AH26} and Lemma~\ref{lem:fixed}(ii) descend normality.  Normality recombines across a central direct sum.
\end{proof}

\begin{proposition}\label{prop:mc-descent}  
    If $R_{\bbC}$ is monotone complete, then $R$ is monotone complete.
\end{proposition}

\begin{proof}
    Canonical conjugation on $R_{\bbC}$ is a real $*$-automorphism with fixed-point algebra $R$, so Lemma~\ref{lem:fixed}(iii) applies.
\end{proof}

\begin{theorem}\label{thm:real-complex-mc}
    ZFC proves that every real \AWstar-algebra is monotone complete if and only if every complex \AWstar-algebra is monotone complete.
\end{theorem}

\begin{proof}
    One direction is immediate because a complex \AWstar-algebra and its underlying real algebra have the same self-adjoint part and order.  Conversely assume every complex \AWstar-algebra is monotone complete and let $R$ be real \AWstar.  Split $R = zR\oplus(1 - z)R$.  The semi-abelian summand is monotone complete by Lemma~\ref{lem:semi-selfadjoint}.  The complementary summand has monotone-complete complexification by Theorem~\ref{thm:handelman} and the hypothesis; Proposition~\ref{prop:mc-descent} descends the supremum structure.
\end{proof}

\begin{corollary}\label{cor:finite-real-mc}
    ZFC proves that the following are equivalent:
    \begin{enumerate}
        \item[(i)] every finite real \AWstar-algebra is monotone complete;
        \item[(ii)] every finite complex \AWstar-algebra is monotone complete;
        \item[(iii)] $\sigma_{\mathrm{MC}}$.
    \end{enumerate}
    Consequently the finite real monotone-completeness conjecture has the same truth value in every $\omega$-model of ZFC.
\end{corollary}

\begin{proof}
    The equivalence of (ii) and (iii) is Theorem~\ref{thm:finite-algebras}.  The implication (i)$\Rightarrow$(ii) again uses the common self-adjoint order.  Conversely split a finite real $R$ by Lemma~\ref{lem:semi-split}.  The semi-abelian part is monotone complete; the complementary complexification is finite by type comparison and hence monotone complete under (ii), so Proposition~\ref{prop:mc-descent} applies.
\end{proof}

\subsection{Matrix stability and a quaternionic counterexample}

\begin{theorem}[Matrix-stability criterion]\label{thm:matrix-stability}
    Let $z$ be the maximal central semi-abelian projection of a real \AWstar-algebra $R$.  The following are equivalent:
    \begin{enumerate}
        \item[(i)] $R_{\bbC}$ is a complex \AWstar-algebra;
        \item[(ii)] $M_n(R)$ is a real \AWstar-algebra for some $n \geq 2$;
        \item[(iii)] $M_n(R)$ is a real \AWstar-algebra for every $n \geq 2$;
        \item[(iv)] $(zR)_{\bbC}$ is a complex \AWstar-algebra.
    \end{enumerate}
    Thus the obstruction to complexification and matrix stability is concentrated in the central semi-abelian summand, which need not be commutative.
\end{theorem}

\begin{proof}
    (i)$\Rightarrow$(iii): each $M_n(R_{\bbC})$ is complex \AWstar, by Berberian on the finite central summand and matrix stability of the properly infinite summand~\cite{Ber58,Ber72}.  Entrywise conjugation has fixed-point algebra $M_n(R)$, so Lemma~\ref{lem:fixed}(i) applies.  (iii)$\Rightarrow$(ii) is immediate.  If (ii) holds, the corner of $M_n(R)$ cut out by $e_{11} + e_{22}$ is a real \AWstar-algebra naturally isomorphic to $M_2(R)$, and Proposition~\ref{prop:m2-complex} gives (i).
    
    For (i)$\Leftrightarrow$(iv), split $R = zR\oplus(1 - z)R$.  By Lemma~\ref{lem:semi-split} and Theorem~\ref{thm:handelman}, $((1 - z)R)_{\bbC}$ is always complex \AWstar.  The \AWstar property passes to and recombines across central direct summands.
\end{proof}

The next example shows that the semi-abelian obstruction is genuinely noncommutative.

\begin{proposition}[Quaternionic sequence algebra]\label{prop:quat-example}
    Let
    \[
        R = \{(q_n) \in \ell^{\infty}(\bbN,\mathbb{H}):\operatorname{dist}(q_n,\bbR) \to 0\} = \ell^{\infty}(\bbN,\bbR) + c_0(\bbN,\mathbb{H}).
    \]
    Then:
    \begin{enumerate}
        \item[(i)] $R$ is a finite semi-abelian real \AWstar-algebra, $R_{\sa} = Z(R) = \ell^{\infty}(\bbN,\bbR)$, and no nonzero central summand of $R$ is commutative;
        \item[(ii)] under $\mathbb{H}_{\bbC} \cong  M_2(\bbC)$,
        \[
         R_{\bbC} \cong \ell^{\infty}(\bbN,\bbC)I_2 + c_0(\bbN,M_2(\bbC)),
        \]
        and this algebra is not even Rickart;
        \item[(iii)] fixing a quaternionic imaginary unit $I$, the map
        \[
         \delta(q)_n = Iq_n - q_nI
        \]
        is a bounded outer $*$-derivation of $R$;
        \item[(iv)] $Z(R)$ is a real \wstar-algebra, but $R$ is not a real \wstar-algebra.
    \end{enumerate}
    In particular, $R$ is type I in the real Baer classification and gives a negative answer to the problem of removing the complexification hypothesis from Ayupov's real type I \wstar-center criterion~\cite{Ayu04typeI}; it also gives a negative answer to the universal innerness question for derivations raised in~\cite{DR21}.
\end{proposition}

\begin{proof}
    The algebra is a norm-closed unital real \cstar-subalgebra of $\prod_n\mathbb{H}$.  A self-adjoint quaternion is real.  Commuting with every element supported at the $n$th coordinate forces the $n$th coordinate to be real, so $R_{\sa} = Z(R) = \ell^{\infty}(\bbN,\bbR)$.  The projections are exactly the $0$--$1$ sequences and are central.
    
    For $S\subseteq R$ let
    \[
     E = \{n:s_n = 0\text{ for every }s\in S\},\qquad p = 1_E.
    \]
    Since $\mathbb{H}$ is a division algebra, $sx = 0$ for every $s\in S$ exactly when $x_n = 0$ outside $E$.  Hence $\operatorname{ann}_r(S) = pR$, so $R$ is \AWstar.  It is finite by Lemma~\ref{lem:semi-selfadjoint}.  Every nonzero central projection dominates some coordinate projection $e_n$, and $e_nR \cong \mathbb{H}$ is noncommutative; hence no nonzero central summand is commutative.
    
    Complexifying coordinatewise gives the displayed algebra $A = R_{\bbC}$.  Let $a\in A$ have $a_n = n^{-1}e_{11}$.  If $\operatorname{ann}_r^A(a) = pA$ for a projection $p$, then $e_{11}p_n = 0$ for every $n$.  The element supported at coordinate $n$ with value $e_{22}$ belongs to $\operatorname{ann}_r^A(a)$, so $p_ne_{22} = e_{22}$.  These relations force $p_n = e_{22}$ for every $n$.  But the constant sequence $e_{22}$ is not in $A$, since $\operatorname{dist}(e_{22},\bbC I_2) = 1/2$.  Thus $A$ is not Rickart.
    
    For the derivation, $\|Iq_n-q_nI\| \leq 2\operatorname{dist}(q_n,\bbR)$, so $\delta(R)\subseteq c_0(\bbN,\mathbb{H})$.  It is a bounded real-linear $*$-derivation.  If $\delta = [a,\cdot]$ with $a\in R$, testing against all quaternions supported at coordinate $n$ shows that $a_n-I$ commutes with $\mathbb{H}$, hence $a_n-I\in\bbR$.  Therefore $\operatorname{dist}(a_n,\bbR) = 1$ for every $n$, contradicting $a\in R$.
    
    Finally $Z(R) = \ell^{\infty}(\bbN,\bbR)$ is a real \wstar-algebra.  If $R$ were a real \wstar-algebra, the bounded net
    \[
     s_F = \sum_{n\in F}e_nI\qquad(F\subseteq\bbN\text{ finite})
    \]
    would have an ultraweakly convergent subnet.  If its limit were $s\in R$, ultraweak continuity of multiplication by each $e_n$ would give $e_ns = e_nI$ for every $n$, so $s_n = I$ for all $n$, contrary to the definition of $R$.
\end{proof}

\begin{corollary}\label{cor:typeI-corrected}
    Let $R$ be a real type I \AWstar-algebra.  Then
    \[
     R\text{ is a real }\mathrm{W}^*\text{-algebra}
    \]
    if and only if $Z(R)$ is a real \wstar-algebra and $R$ has bounded central mixing over $\Proj(Z(R)_{\sa})$.
\end{corollary}

\begin{proof}
    A real \wstar-algebra has \wstar-complexification and therefore bounded central mixing.  Conversely, bounded central mixing is equivalent to $R_{\bbC}$ being \AWstar by Theorem~\ref{thm:central-mixing} below.  Ayupov's type I theorem then applies~\cite{Ayu04typeI}.
\end{proof}

\subsection{Central mixing and real Boolean-valued reduction}

Put $Z_0 = Z(R)_{\sa}$ and $\mathcal{B} = \Proj(Z_0)$.

\begin{lemma}\label{lem:z0}
    The algebra $Z_0$ is a commutative real \AWstar-algebra.  If $X$ is the Stone space of $\mathcal{B}$, then $X$ is Stonean and $Z_0 \cong  C(X,\bbR)$.  The Boolean algebra $\mathcal{B}$ is complete, and joins of families in $\mathcal{B}$ agree whether computed in $Z_0$ or in $R$.
\end{lemma}

\begin{proof}
    If $S\subseteq Z_0$ and $\operatorname{ann}_r^R(S) = eR$, centrality of $S$ makes $eR$ a two-sided ideal, hence $e$ is central and $\operatorname{ann}_r^{Z_0}(S) = eZ_0$.  Thus $Z_0$ is commutative real \AWstar and real Gelfand theory gives $C(X,\bbR)$ with $X$ Stonean~\cite{Li03}.  The same two-sided-ideal argument identifies projection joins in $Z_0$ and $R$.
\end{proof}

\begin{definition}
    We say that $R$ has \emph{bounded central mixing over $\mathcal{B}$} if, whenever $(b_i)_{i\in I}$ is a partition of $1$ in $\mathcal{B}$ and $(x_i)_{i\in I}\subseteq R$ is uniformly norm bounded, there is $x\in R$ such that $b_ix = b_ix_i$ for all $i$.
\end{definition}

\begin{lemma}\label{lem:complex-mixing}
    If $R_{\bbC}$ is a complex \AWstar-algebra, then $R$ has bounded central mixing over $\mathcal{B}$.
\end{lemma}

\begin{proof}
    The join of a central partition from $R$ remains $1$ in $R_{\bbC}$.  The standard central mixing theorem in a complex \AWstar-algebra~\cite{Ber72} gives a mixture $x \in R_{\bbC}$.  Canonical conjugation gives a second mixture with the same pieces; uniqueness of central mixtures forces $\bar x = x$, so $x \in R$.
\end{proof}

\begin{proposition}[Bounded-section construction]\label{prop:bounded-section}
    If $R$ has bounded central mixing over $\mathcal{B}$, there is an internal unital real \cstar-algebra $\mathcal{M}$ in the Boolean-valued universe $V^{\mathcal{B}}$ and a unital $Z_0$-linear real $*$-isomorphism
    \[
     R \cong \operatorname{bd}(\mathcal{M}),
    \]
    where $\operatorname{bd}(\mathcal{M})$ is the bounded global-section algebra.  Moreover
    \[
     x\in Z(R)\Longleftrightarrow \bv{\widehat x\in Z(\mathcal{M})} = 1,
     \qquad
     \bv{Z(\mathcal{M})_{\sa} = \bbR1} = 1.
    \]
\end{proposition}

\begin{proof}
    This is the real form of Ozawa's bounded-section construction, compare Katz~\cite{Oz86,Katz13}.  For completeness, define the $Z_0$-valued norm
    \[
     \|x\|_{Z_0} = \inf\{a\in (Z_0)_{+} \ : \ x^*x \leq a^2\}.
    \]
    Lemma~\ref{lem:z0} gives the infimum.  We first check that this infimum is attained in the order-theoretic sense needed below.  Fix $x\in R$ and put
    \[
        S_x = \{a\in (Z_0)_{+} \ : \ x^*x \leq a^2\},\qquad c = \inf S_x.
    \]
    We claim that $x^*x \leq c^2$.  Since $c$ is a lower bound for $S_x$, we have $c \leq a$ for every $a\in S_x$.

    Fix $\varepsilon>0$.  For $a\in S_x$, let
    \[
        b_a = E_{[0,\varepsilon]}(a-c)\in\mathcal B,
    \]
    the spectral projection of the positive central element $a-c$ corresponding to $[0,\varepsilon]$.  We claim that
    \[
        \bigvee_{a\in S_x}b_a = 1.
    \]
    Otherwise, with
    \[
        0 \not= e = 1-\bigvee_{a\in S_x}b_a,
    \]
    we have $e(a - c) \geq \varepsilon e$ for every $a \in S_x$.  Hence
    \[
        c + \varepsilon e \leq a \qquad (a \in S_x),
    \]
    because this inequality holds on the band $e$, while on $1 - e$ it reduces to $c \leq a$.  Thus $c + \varepsilon e$ is a lower bound for $S_x$ strictly larger than $c$, contradicting the definition of $c$ as the greatest lower bound.

    Refine the cover $(b_a)_{a \in S_x}$ to a central partition $(e_i)_{i \in I}$ of $1$, with $e_i \leq b_{a_i}$ for suitable $a_i \in S_x$.  Then
    \[
        0 \leq e_i(a_i - c) \leq \varepsilon e_i,
    \]
    and therefore
    \[
        e_i a_i \leq e_i(c + \varepsilon).
    \]
    The family $(e_i a_i)_{i \in I}$ is uniformly bounded.  By bounded central mixing there is $a_{\varepsilon} \in Z_0$ such that
    \[
        e_i a_{\varepsilon} = e_i a_i\qquad(i\in I).
    \]
    Since $a_i \in S_x$, bandwise
    \[
        e_i x^*x \leq e_i a_i^2 = e_i a_{\varepsilon}^2,
    \]
    whence $x^*x \leq a_{\varepsilon}^2$.  Also $a_{\varepsilon} \leq c + \varepsilon$, so
    \[
        x^*x \leq (c + \varepsilon)^2.
    \]
    Letting $\varepsilon \downarrow 0$ and using norm closedness of the positive cone gives $x^*x \leq c^2$.  Thus the infimum defining $\|x\|_{Z_0}$ has the required order property.

    Bandwise inequalities can therefore be checked on a central partition
    and glued back in $R$, while bounded central mixing makes the usual
    mixing completion complete.  The resulting internal norm satisfies the real \cstar identities and its bounded descent is exactly $R$.
    
    Centrality ascends and descends by rescaling arbitrary internal elements into bounded ones.  A bounded internally central self-adjoint element descends to $Z_0$, which is precisely the bounded descent of the internal real scalars; rescaling removes the boundedness restriction and gives the displayed self-adjoint-center identity.
\end{proof}

\begin{proposition}\label{prop:aw-ascent}
    Let $\mathcal{M}$ be an internal unital real \cstar-algebra in $V^{\mathcal{B}}$ and put $A = \operatorname{bd}(\mathcal{M})$.  Then
    \[
     A \text{ is real }\mathrm{AW}^* \Longleftrightarrow \bv{\mathcal{M} \text{ is real }\mathrm{AW}^*} = 1.
    \]  
\end{proposition}

\begin{proof}
    Ozawa's annihilator argument is scalar-field independent~\cite{Oz86}.  If $\mathcal{M}$ is internally \AWstar, ascend an external subset of $A$ and descend the internal projection generating its annihilator.  Conversely normalize an arbitrary internal subset without changing its annihilator, descend all bounded members to $A$, and ascend the projection generating the external annihilator.
\end{proof}

\begin{lemma}[Internal center dichotomy]\label{lem:center-dichotomy}
    Let $N$ be a unital real \cstar-algebra with $Z(N)_{\sa} = \bbR 1$.  Then either $Z(N) = \bbR 1$, or there is a central element $J$ with
    \[
     J^* = -J,\qquad J^2 = -1,\qquad Z(N) = \bbR 1 \oplus \bbR J \cong \bbC.
    \]
    If $N$ is real \AWstar, then in the second case multiplication by $J$ makes $N$ a complex \AWstar-factor.
\end{lemma}

\begin{proof}
    A nonzero central skew-adjoint element $x$ has $x^*x = -x^2 \in \bbR 1$; after normalization it gives $J$.  Any other central skew-adjoint element $y$ has $Jy$ central self-adjoint and hence scalar, so $y$ is a real multiple of $J$.  Declaring $i = J$ does not alter right annihilators.
\end{proof}

\begin{theorem}\label{thm:central-mixing}
    Let $R$ be a real \AWstar-algebra.  The following are equivalent:
    \begin{enumerate}
        \item[(i)] $R_{\bbC}$ is a complex \AWstar-algebra;
        \item[(ii)] $M_2(R)$ is a real \AWstar-algebra;
        \item[(iii)] $R$ has bounded central mixing over $\mathcal{B}$;
        \item[(iv)] in $V^{\mathcal{B}}$ there is an internal real \AWstar-algebra $\mathcal{M}$ such that $R \cong \operatorname{bd}(\mathcal{M})$ $Z_0$-linearly and $\bv{Z(\mathcal{M})_{\sa} = \bbR1} = 1$.
    \end{enumerate}
\end{theorem}

\begin{proof}
    Proposition~\ref{prop:m2-complex} gives (i)$\Leftrightarrow$(ii), Lemma~\ref{lem:complex-mixing} gives (i)$\Rightarrow$(iii), and Propositions~\ref{prop:bounded-section}--\ref{prop:aw-ascent} give (iii)$\Rightarrow$(iv).  The mixing principle gives (iv)$\Rightarrow$(iii).
    
    Assume (iv).  Apply Lemma~\ref{lem:center-dichotomy} internally.  On the band where $Z(\mathcal{M}) = \bbR1$, Corollary~\ref{cor:factor-complex} transfers internally and makes $\mathcal{M}_{\bbC}$ a complex \AWstar-factor.  On the complementary band $\mathcal{M}$ itself is a complex \AWstar-factor when scalar multiplication is defined by the central $J$, and its real complexification is the direct sum of that factor and its conjugate.  Thus $\bv{\mathcal{M}_{\bbC}\text{ is complex }\mathrm{AW}^*} = 1$.  Ozawa's bounded-section theorem and the identity
    \[
     \operatorname{bd}(\mathcal{M}_{\bbC}) \cong \operatorname{bd}(\mathcal{M})_{\bbC}
    \]
    now give (i).
\end{proof}

\begin{corollary}\label{cor:center-split}
    Under the equivalent conditions of Theorem~\ref{thm:central-mixing}, there is a central projection $c$ such that $Z(cR) = Z(cR)_{\sa}$, while $(1 - c)R$ contains a central $J_0$ with
    \[
     J_0^* = -J_0,\qquad J_0^2 = -(1 - c).
    \]
    Thus $(1 - c)R$, with $ix = J_0x$, is a complex \AWstar-algebra whose underlying real algebra is $(1 - c)R$.
\end{corollary}

\begin{proof}
    Take $c = \bv{Z(\mathcal{M}) = \bbR 1}$ in the Boolean-valued representation and descend the central imaginary unit on the complementary band.
\end{proof}

\begin{proposition}\label{prop:finite-internal}
    In Theorem~\ref{thm:central-mixing}, if $R$ is finite then $\bv{\mathcal{M}\text{ is finite}} = 1$.
\end{proposition}

\begin{proof}
    If a nonzero Boolean band forced a proper internal isometry, the maximum and mixing principles would produce a globally defined internal isometry of norm one which is proper on that band.  It descends to a proper isometry in $R$, contradicting finiteness.
\end{proof}

\subsection{Quasitraces, traces and norm-one center projections}

Write $Z_0(R) = Z(R)_{\sa}$.  We first construct directly the self-adjoint-center-valued quasitrace used below.

\begin{theorem}\label{thm:real-center-quasitrace}
    Let $R$ be a finite real \AWstar-algebra.  There is a unique map
    \[
        T_R : R \longrightarrow Z_0(R)
    \]
    such that $T_R|_{Z_0(R)} = \mathrm{id}$, $T_R(hx) = hT_R(x)$ for all $h \in Z_0(R)$ and $x \in R$, $T_R$ is linear on commutative real \cstar-subalgebras,
    \[
        T_R(x^*x) = T_R(xx^*) \geq 0,
    \]
    with $T_R(x^*x) = 0$ only for $x = 0$, and
    \[
        T_R(a + b) = T_R(a) \qquad (a^* = a,\ b^* = -b).
    \]
\end{theorem}

\begin{proof}
    Split $R = zR \oplus R_1$ by Lemma~\ref{lem:semi-split}, and put
    \[
        A = (R_1)_{\bbC}.
    \]
    By Theorem~\ref{thm:handelman} and real/complex type comparison, $A$ is a finite complex \AWstar-algebra.  Let $T_A : A \to Z(A)$ be its canonical normalized center-valued quasitrace.  If $\gamma$ denotes the canonical conjugation on $A$, uniqueness of $T_A$ gives
    \[
        T_A(\gamma(a)) = \gamma(T_A(a)).
    \]
    Define
    \[
        T_R(x_0 \oplus x_1) = \rea x_0 \oplus T_A(\rea x_1)
        \qquad(x_0 \in zR,\ x_1 \in R_1).
    \]
    By Lemma~\ref{lem:semi-selfadjoint}, $\rea x_0$ lies in $Z(zR)_{\sa}$.  Also $\rea x_1$ is fixed by $\gamma$, so the equivariance above shows that $T_A(\rea x_1)$ is central, self-adjoint and fixed by $\gamma$, hence belongs to $Z(R_1)_{\sa}$.  Thus $T_R$ takes values in $Z_0(R)$.

    The defining properties now follow bandwise.  On $zR$ they follow from Lemma~\ref{lem:semi-selfadjoint}.  On $R_1$, positivity, faithfulness and the identity at $x^*x$ and $xx^*$ follow from the corresponding properties of $T_A$.  If $C \subseteq R_1$ is a commutative real \cstar-subalgebra, then its complexification inside $A$ is commutative, so linearity of $T_A$ there gives real linearity of $T_R$ on $C$.  Central homogeneity follows in the same way, and the displayed skew-annihilation identity is immediate from the definition.

    It remains to prove uniqueness.  Let $S : R \to Z_0(R)$ be another map with these properties.  On the semi-abelian summand $zR$, Lemma~\ref{lem:semi-selfadjoint} gives
    \[
        S(x) = S(\rea x) = \rea x,
    \]
    so $S = T_R$ there.

    On $R_1$, consider the restrictions of $S$ and $T_R$ to projections.  Both are positive, normalized, finitely additive on orthogonal projections, invariant under equivalence, and compatible with central bands.  Through the Boolean-valued representation of Theorem~\ref{thm:central-mixing} and Proposition~\ref{prop:finite-internal}, these properties reduce on each Boolean band to the usual normalized projection dimension in a finite factor.  On the real-center band the internal algebra is a finite real factor; on the complementary band its central imaginary unit makes it a finite complex factor.  In finite type $\mathrm{I}$, equivalence of minimal projections forces the value $\operatorname{rank}(p)/n$, with rank over $\bbR$, $\bbC$ or $\mathbb H$ as appropriate.  In type $\mathrm{II}_1$, comparison with projections of rational dimensions $m/n$ and $(m + 1)/n$ and then letting $n\to\infty$ forces the usual dimension.  Hence $S$ and $T_R$ agree on every projection of $R_1$.

    Every self-adjoint element of $R_1$ is uniformly approximated inside its abelian \AWstar-subalgebra by finite real linear combinations of spectral projections.  Since both maps are linear and norm continuous on that abelian algebra, they agree on all self-adjoint elements.  Both maps discard the skew-adjoint part, so they agree on all of $R_1$.  This proves uniqueness.
\end{proof}

Rakhimov and Rakhmonova state a related full-center-valued construction for finite real \AWstar-algebras~\cite{RR24}; the formulation above is self-contained and uses the self-adjoint center, which is the form needed below.

\begin{lemma}\label{lem:semi-trace}
    Let $B$ be a semi-abelian real \AWstar-algebra.
    \begin{enumerate}
        \item[(i)] $T_B(x) = \rea x = (x + x^*)/2$ for the canonical $Z(B)_{\sa}$-valued quasitrace.
        \item[(ii)] Every scalar $1$-quasitrace on $B$ is a real-linear trace.
        \item[(iii)] There is a faithful norm-one center-valued trace $E_B : B \to Z(B)$.
    \end{enumerate}
\end{lemma}

\begin{proof}
    Part (i) follows from Lemma~\ref{lem:semi-selfadjoint}: $\rea x\in Z(B)_{\sa}$, and the defining properties of $T_B$ force $T_B(x) = \rea x$.
    
    For skew-adjoint $k,\ell$ put
    \[
     \langle k,\ell\rangle = -\frac{1}{2}(k\ell + \ell k)\in Z(B)_{\sa}.
    \]
    If $x = a + k$ and $y = b + \ell$ with $a,b \in B_{\sa} = Z(B)_{\sa}$ and $k,\ell$ skew-adjoint, then
    \[
     \rea(xy) = ab + \frac{1}{2}(k\ell + \ell k) = \rea(yx). \tag{7.1}
    \]
    A scalar $1$-quasitrace annihilates skew-adjoint elements and is linear on the central self-adjoint part, so it is real-linear on $B$ and (7.1) makes it tracial.  This proves (ii).
    
    For (iii), let $c$ be the largest central projection such that $cB$ is commutative.  Such a largest projection exists because a supremum of central commutative bands remains commutative: each commutator is annihilated by every band in the family and hence by their join.  Put $d = 1 - c$.
    
    We claim that $Z(dB) = Z(dB)_{\sa}$.  Otherwise choose a nonzero central skew-adjoint $s\in dB$.  On a nonzero spectral band $e$ of $|s|$ the element $s$ is bounded away from zero, and
    \[
        J = s|s|^{-1}\in Z(eB),\qquad J^* = -J,\quad J^2 = -e.
    \]
    Declaring $ix = Jx$ makes $eB$ a complex \AWstar-algebra without changing projections.  It is semi-abelian as a complex \AWstar-algebra, hence commutative, contradicting maximality of $c$.
    
    Define
    \[
        E_B(x) = cx + d\rea x.
    \]
    On $cB$ this is the identity; on $dB$ the center is self-adjoint, so $E_B$ is a $Z(B)$-module map.  Equation (7.1) makes it tracial.  It is positive and unital.  Since the two central bands are orthogonal,
    \[
        \|E_B(x)\| = \max\{\|cx\|,\|d\rea x\|\} \leq \|x\|,
    \]
    so $E_B$ is contractive and hence has norm one.  If $E_B(x^*x) = 0$, then
    $cx^*x = 0$ and $dx^*x = 0$, whence $cx = dx = 0$.  Thus it is faithful.
\end{proof}

Rakhimov and Rakhmonova state the corresponding factorization under an additional complexification hypothesis~\cite[Theorem~6]{RR24}.  The following proof is direct: it works on projections by dimension comparison and therefore does not complexify a nonlinear scalar quasitrace.

\begin{theorem}\label{thm:scalar-factorization}
    Let $R$ be a finite real \AWstar-algebra, let $T_R : R \to Z_0(R)$ be its center-valued quasitrace, and let $\tau$ be a scalar $1$-quasitrace on $R$.  Then there is a unique positive linear functional $\varphi$ on $Z_0(R)$ such that
    \[
        \tau = \varphi \circ T_R.
    \]
\end{theorem}

\begin{proof}
    Split $R = zR\oplus R_1$ by Lemma~\ref{lem:semi-split}.  On the semi-abelian summand, Lemma~\ref{lem:semi-trace}(i) gives $T_{zR}(x) = \rea x$, while every scalar quasitrace satisfies $\tau(x) = \tau(\rea x)$.  Thus factorization holds there with
    \[
        \varphi_0 = \tau|_{Z_0(zR)}.
    \]

    It remains to treat $R_1 = (1 - z)R$.  Put
    \[
        D(p) = T_{R_1}(p) \in Z_0(R_1),\qquad
        \varphi_1 = \tau|_{Z_0(R_1)}.
    \]
    We first prove
    \[
        \tau(p) = \varphi_1(D(p))
    \]
    for every projection $p \in R_1$.

    Fix an integer $N \geq 2$.  Use the Boolean-valued representation of Theorem~\ref{thm:central-mixing} and Proposition~\ref{prop:finite-internal}.  On the real-center Boolean band the internal algebra is a finite real factor.  On the complementary band its central imaginary unit makes it a finite complex factor.  In either case it is of type $\mathrm{II}_1$ or finite type $\mathrm{I}$; in finite type $\mathrm{I}$, the rank below is the matrix rank over $\bbR$, $\bbC$, or $\mathbb H$, as appropriate.  Split off the finitely many central type-$\mathrm{I}$ rank bands of rank $n < N^2$, and refine each such band according to the rank $k$ of $p$.  On a rank-$n$ band $e$, equivalence of minimal projections and finite additivity give
    \[
        \tau(ep) = \frac{k}{n} \tau(e) = \varphi_1\!\left(\frac{k}{n}e\right)
        = \varphi_1(eD(p)).
    \]
    Thus the projection identity is exact on these bands.

    Let $h$ be the complementary central projection.  On every factor over $h$, choose $N$ equivalent mutually orthogonal projections $q_1,\ldots,q_N$ and a remainder $r$ such that
    \[
        h = q_1 + \cdots + q_N + r,\qquad r \precsim q_1.
    \]
    In a type-$\mathrm{II}_1$ factor take an exact equal partition and $r = 0$; in a type-$\mathrm{I}$ factor of rank $n \geq N^2$, take each $q_i$ of rank $\lfloor n/N\rfloor$.  The bounded central mixing already established lets these choices descend simultaneously to $hR_1$.

    For every central projection $e \leq h$, equivalence and finite additivity give
    \[
        \frac{\varphi_1(e)}{N + 1} \leq \tau(eq_i) \leq \frac{\varphi_1(e)}{N}, \qquad \frac{e}{N + 1} \leq eD(q_i) \leq \frac{e}{N}.
    \]
    Partition $h$ by the central spectral projections $e_j$, $0 \leq j < N$, of $D(p)$ so that
    \[
        \frac{j}{N}e_j \leq e_jD(p) \leq \frac{j + 1}{N}e_j.
    \]
    Put $Q_k = q_1 + \cdots + q_k$ and $Q_0 = 0$.  By projection comparison and the preceding bounds,
    \[
        e_j Q_j \precsim e_jp,
    \]
    while for $j \leq N - 2$,
    \[
        e_j p \precsim e_j Q_{j + 2},
    \]
    because
    \[
        \frac{j + 2}{N + 1} \geq \frac{j + 1}{N}.
    \]
    For the last bin we use only $e_{N-1}p \leq e_{N-1}$.  Combining these comparison inequalities with the preceding bounds and the bounds defining the bins gives, for every $j$,
    \[
        \bigl|\tau(e_jp) - \varphi_1(e_jD(p))\bigr| \leq \frac{2\varphi_1(e_j)}{N}.
    \]
    Summing over the finitely many bins and adding the low-rank bands, on which equality is exact, yields
    \[
        \bigl|\tau(p) - \varphi_1(D(p))\bigr| \leq \frac{2\varphi_1(1)}{N}.
    \]
    Since $N$ was arbitrary, the projection identity follows.

    Every self-adjoint element is uniformly approximated, inside the abelian \AWstar-algebra generated by it, by finite real linear combinations of spectral projections.  Both $\tau$ and $\varphi_1 \circ T_{R_1}$ are linear and norm continuous on that abelian algebra, so the projection identity extends to every self-adjoint element.  Both maps discard the skew-adjoint part, hence
    \[
        \tau = \varphi_1 \circ T_{R_1}
    \]
    on all of $R_1$.  Taking $\varphi = \varphi_0 \oplus \varphi_1$ proves existence.  Uniqueness follows from $T_R|_{Z_0(R)} = \mathrm{id}$.
\end{proof}

\begin{theorem}[Matrix extensions]\label{thm:matrix-extension}
    Let $R$ be a finite real \AWstar-algebra and let $\tau$ be a scalar $1$-quasitrace on $R$.  For every $n \geq 1$ there is a scalar $1$-quasitrace $\tau_n$ on $M_n(R)$ such that
    \[
        \tau_n(x \otimes e_{11}) = \tau(x).
    \]
    Thus every scalar $1$-quasitrace on a finite real \AWstar-algebra is an $n$-quasitrace for every $n$.
\end{theorem}

\begin{proof}
    On the semi-abelian central summand $B = zR$, Lemma~\ref{lem:semi-trace}(ii) makes $\tau|_B$ a real trace, so
    \[
        \tau_{B,n}([x_{ij}]) = \sum_{i = 1}^n\tau(x_{ii})
    \]
    is a positive real trace on $M_n(B)$ extending the $(1,1)$ corner.
    
    On $R_1 = (1 - z)R$, Theorem~\ref{thm:scalar-factorization} gives $\tau_1 = \varphi \circ T_{R_1}$.  Theorem~\ref{thm:matrix-stability} makes $M_n(R_1)$ a real \AWstar-algebra.  Since $R_1$ is finite, type comparison makes $(R_1)_{\bbC}$ finite; hence $M_n((R_1)_{\bbC})$ is finite.  Its fixed-point algebra under entrywise conjugation is $M_n(R_1)$, so $M_n(R_1)$ is finite as well.  Let $T_n$ be its normalized center-valued quasitrace, identify $Z(M_n(R_1)) = Z(R_1)I_n$, and write $T_n(X) = E_n(X) I_n$.  The diagonal matrix projections are equivalent and sum to $I_n$, so $E_n(e_{ii}) = 1/n$.  By uniqueness of the normalized center-valued quasitrace,
    \[
        x \longmapsto nE_n(x \otimes e_{11})
    \]
    is $T_{R_1}$.  Hence
    \[
        \tau_{1,n}(X) = n \varphi(E_n(X))
    \]
    extends $\tau_1$.  Taking the direct sum of the two matrix quasitraces finishes the proof.
\end{proof}

A center-valued trace on a finite real \AWstar-algebra means a faithful positive real-linear map $T:R\to Z(R)$ satisfying
\[
    T(1) = 1,\qquad T(zx) = zT(x),\qquad T(xy) = T(yx).
\]

\begin{theorem}\label{thm:real-trace-arithmetic}
    ZFC proves that the following are equivalent:
    \begin{enumerate}
        \item[(i)] $\sigma_K$;
        \item[(ii)] every finite real \AWstar-algebra admits a center-valued trace;
        \item[(iii)] for every finite real \AWstar-algebra $R$, its canonical $Z(R)_{\sa}$-valued quasitrace $T_R$ is real-linear.
    \end{enumerate}
    Consequently these assertions have the same truth value in every $\omega$-model of ZFC.
\end{theorem}

\begin{proof}
    Suppose first that $R$ has a center-valued trace $T$.  Then
    \[
        S(x) = T\! \left(\frac{x + x^*}{2}\right)
    \]
    is a real-linear $Z_0(R)$-valued quasitrace with the defining properties of $T_R$, so uniqueness gives $S = T_R$.  Thus (ii)$\Rightarrow$(iii).
    
    Conversely assume $T_R$ is real-linear.  Split $R = zR \oplus R_1$ by Lemma~\ref{lem:semi-split}.  Lemma~\ref{lem:semi-trace}(iii) gives a center-valued trace on $zR$.  The algebra $R_1$ satisfies Theorem~\ref{thm:central-mixing}; let $c$ and $J_0$ be as in Corollary~\ref{cor:center-split}.  On $cR_1$, the center is self-adjoint and the real-linear quasitrace $T_R$ is already a center-valued trace: the identity $T_R(u^*u) = T_R(uu^*)$ polarizes to traciality.  On $(1 - c)R_1$ define
    \[
        T_1(x) = T_R(x) - J_0 T_R(J_0 x).
    \]
    Real-linearity, central homogeneity and $J_0^2 = - (1 - c)$ make $T_1$ a full center-module map; on positive elements the second term vanishes because $J_0 x$ is skew-adjoint, so $T_1$ is positive and faithful.  The same polarization gives traciality.  This proves (iii)$\Rightarrow$(ii).
    
    Assume $\sigma_K$ and let $R$ be finite.  On the semi-abelian central summand use Lemma~\ref{lem:semi-trace}(iii).  On the complement use the Boolean-valued representation from Theorem~\ref{thm:central-mixing}; Proposition~\ref{prop:finite-internal} makes the internal algebra finite.  On the internal complex-factor band, $\sigma_K$ makes the factor a von Neumann factor and supplies its trace.  On the real-factor band, Corollary~\ref{cor:factor-complex} gives a finite complex \AWstar-factor, $\sigma_K$ makes it von Neumann, and uniqueness of its normalized trace makes that trace conjugation invariant and hence real-valued on the real form.  Mixing the two traces and descending gives a center-valued trace.  Thus (i)$\Rightarrow$(ii).
    
    Finally assume (ii) and let $A$ be a finite complex \AWstar-factor, regarded as a real \AWstar-algebra.  A center-valued trace $T : A \to Z_{\bbR}(A) = \bbC 1$ is a full-center module map, hence complex-linear because $i1$ is central.  It therefore gives a faithful complex trace on $A$.  Wright's criterion makes $A$ a von Neumann algebra~\cite{Wri76}, so Kaplansky's problem has a positive answer and Theorem~\ref{thm:absoluteness} gives $\sigma_K$.
\end{proof}

The next theorem recovers the norm-one projection characterization without using an unrestricted real Tomiyama theorem.

\begin{lemma}\label{lem:center-projection}
    Let $R$ be a real \AWstar-algebra and let $P : R \to Z(R)$ be a real-linear projection of norm one.  Then for every central projection $e$,
    \[
        P(ex) = eP(x),
    \]
    and consequently
    \[
        P(hx) = hP(x) \qquad (h \in Z(R)_{\sa}).
    \]
\end{lemma}

\begin{proof}
    It is enough to take $x\in eR$ with $\|x\| \leq 1$.  Since $e$ is central,
    \[
        \|x\pm(1 - e)\| \leq 1.
    \]
    Therefore
    \[
        \|P(x)\pm(1 - e)\| \leq 1.
    \]
    Put $y = (1 - e)P(x)$.  In the commutative real \cstar-algebra $(1 - e)Z(R)$ we have $\|1 - e\pm y\| \leq 1$.  Hence
    \[
        (1 - e\pm y)^*(1 - e\pm y) \leq 1 - e.
    \]
    Adding the two inequalities yields $2(1 - e + y^*y) \leq 2(1 - e)$, so $y = 0$.  Thus $P(x) = eP(x)$, proving the projection identity.  Every self-adjoint central element is a norm limit of real linear combinations of central projections, giving the second assertion.
\end{proof}

Here a positive map $F : R \to Z(R)$ is called \emph{strictly positive} if $F(a) \not= 0$ for every nonzero $a \in R_{+}$, and a positive map $P : R \to Z(R)$ is called \emph{faithful} if $P(x^*x) = 0$ implies $x = 0$.

\begin{theorem}\label{thm:norm-one}
    Let $R$ be a finite real \AWstar-algebra.  The following are equivalent:
    \begin{enumerate}
        \item[(i)] $R$ has a center-valued trace;
        \item[(ii)] there is a strictly positive real-linear center-module map $F : R \to Z(R)$;
        \item[(iii)] $Z(R)$ is the range of a faithful norm-one real-linear projection $P : R \to Z(R)$.
    \end{enumerate}
\end{theorem}

\begin{proof}
    The implication (i)$\Rightarrow$(ii) is immediate.
    
    Assume (ii).  Split $R = zR\oplus R_1$ by Lemma~\ref{lem:semi-split}.  On the semi-abelian summand $zR$ use the center-valued trace $E_{zR}$ from Lemma~\ref{lem:semi-trace}.  The complement $R_1$ satisfies Theorem~\ref{thm:central-mixing}; let
    \[
        R_1 = cR_1 \oplus dR_1,\qquad d = 1 - c,
    \]
    be the decomposition supplied by Corollary~\ref{cor:center-split}.  Thus $Z(cR_1) = Z(cR_1)_{\sa}$, while $dR_1$ has a central imaginary unit $J_0$.
    
    On $cR_1$ replace $F$ by its symmetrization
    \[
        F_s(x) = \frac{1}{2} \bigl(F(x) + F(x^*)\bigr) = F(\rea x).
    \]
    It is positive, strictly positive, and $Z(cR_1)$-linear.  If $a = a^*$ and
    $\|a\| \leq 1$, then $-1 \leq a \leq 1$, so positivity gives
    \[
        -F(c) \leq F_s(a) \leq F(c).
    \]
    Consequently
    \[
        \|F_s(x)\| \leq \|F(c)\|\,\|x\| \qquad(x \in cR_1),
    \]
    so $F_s$ is bounded.  On $dR_1$, full center-module linearity gives
    \[
        F(J_0x) = J_0 F(x),
    \]
    so $F|_{dR_1}$ is complex-linear for the complex structure determined by $J_0$.  Being positive, it is bounded.  Thus
    \[
        F_0 = F_s|_{cR_1} \oplus F|_{dR_1}
    \]
    is a bounded positive strictly positive center-module map on $R_1$.
    
    Ascend $F_0$ through Theorem~\ref{thm:central-mixing} to an internal positive center-module map
    \[
        f : \mathcal{M} \longrightarrow Z(\mathcal{M}).
    \]
    Strict positivity transfers bandwise: if a nonzero Boolean band forced a nonzero positive $a$ with $f(a) = 0$, then after restricting to that band and rescaling $a$ to a bounded section, descent would give a nonzero positive element of $R_1$ annihilated by $F_0$.  Proposition~\ref{prop:finite-internal} makes $\mathcal{M}$ finite.  On the internal complex-factor band, $f$ is complex-linear and, after division by the positive scalar $f(1)$, is a faithful complex state; Wright's theorem therefore makes that factor a von Neumann factor.  On the internal real-factor band, first symmetrize,
    \[
        f_s(x) = \frac{1}{2} \bigl(f(x) + f(x^*)\bigr).
    \]
    This agrees with $f$ on positive elements, is faithful, and after division by $f(1)$ is a faithful real state.  Corollary~\ref{cor:finite-faithful-state}, transferred internally, makes the real factor a real von Neumann factor.  Mixing the normalized traces on the two Boolean bands and descending gives a center-valued trace $Q_1$ on $R_1$.  Hence
    \[
        Q = E_{zR} \oplus Q_1
    \]
    is a center-valued trace on $R$.  Since it is obtained from normalized factor traces on every central band, it is contractive; in particular $\|Q\| = 1$.  Therefore $Q$ is also a faithful norm-one projection onto $Z(R)$.  Thus (ii)$\Rightarrow$(i) and (ii)$\Rightarrow$(iii).
    
    Assume (iii).  Lemma~\ref{lem:center-projection} shows that $P$ preserves central bands and is $Z(R)_{\sa}$-linear.  Split off the semi-abelian band, where $E_{zR}$ already gives a center-valued trace.  On $R_1 = (1 - z)R$ ascend $P$ to an internal norm-one real-linear projection
    \[
        p : \mathcal{M} \longrightarrow Z(\mathcal{M}).
    \]
    Faithfulness transfers by the same band-local argument: a nonzero Boolean band forcing $a \not= 0$ and $p(a^*a) = 0$ would, after rescaling to a bounded section and descending, contradict faithfulness of $P$ on the corresponding central band of $R_1$.
    
    On the internal real-factor band, the center is $\bbR1$, so $p$ is a norm-one unital real functional.  For $0 \leq a \leq 1$,
    \[
        |p(a)| \leq 1,\qquad |1 - p(a)| = |p(1 - a)| \leq 1,
    \]
    which forces $0 \leq p(a) \leq 1$.  Thus $p$ is a faithful real state.
    
    On the internal complex-factor band, let $J$ be the central imaginary unit.  If $a = a^*$ and $p(a) = u + Jv$, contractivity applied to $a + tJ$ gives
    \[
        u^2 + (v + t)^2 \leq \|a\|^2 + t^2 \qquad (t \in \bbR),
    \]
    so $v = 0$.  Thus $p(a)$ is real for every self-adjoint $a$.  If in addition $0 \leq a \leq 1$, then
    \[
        |p(a)| \leq 1,\qquad |1 - p(a)| \leq 1,
    \]
    so $0 \leq p(a) \leq 1$.  Hence $\phi = \rea p$ is positive on the self-adjoint part.  If $s^* = -s$ and $p(s) = u + Jv$, the same estimate applied to $s + t1$ gives $u = 0$, so $\phi(s) = 0$.  Define
    \[
     \psi(x) = \phi(x) - i\phi(Jx).
    \]
    Then $\psi$ is complex-linear.  If $a \geq 0$, then $Ja$ is skew-adjoint, so $\psi(a) = \phi(a) = p(a) \geq 0$.  Faithfulness of $p$ gives faithfulness of $\psi$.  Thus $\psi$ is a faithful complex state.
    
    The real-factor band is therefore a real von Neumann factor by Corollary~\ref{cor:finite-faithful-state}, and the complex-factor band is a von Neumann factor by Wright.  Mix their normalized traces and descend to a center-valued trace on $R_1$, then combine it with $E_{zR}$.  This proves (iii)$\Rightarrow$(i), and the three conditions are equivalent.
\end{proof}

\begin{corollary}\label{cor:norm-one-arithmetic}
    ZFC proves that $\sigma_K$ is equivalent to each of the following:
    \begin{enumerate}
        \item[(i)] every finite real \AWstar-algebra has a center-valued trace;
        \item[(ii)] for every finite real \AWstar-algebra $R$, the center is the range of a faithful norm-one real-linear projection $R\to Z(R)$.
    \end{enumerate}
\end{corollary}

\begin{proof}
    Combine Theorems~\ref{thm:real-trace-arithmetic} and~\ref{thm:norm-one}.
\end{proof}

\subsection{Derivations and local derivations}

We first isolate the semi-abelian part of the derivation problem.

\begin{lemma}\label{lem:semi-fibers}
    Let $B$ be a semi-abelian real \AWstar-algebra, identify $B_{\sa} = Z(B)_{\sa} \cong  C(X,\bbR)$, and for $t \in X$ let $I_t$ be the closed ideal generated by the central functions vanishing at $t$.  Then
    \[
        I_t = \{x \in B \ : \ (x^*x)(t) = 0\}.
    \]
    In particular, $I_t$ is proper and the quotients $B_t = B/I_t$ separate points of $B$.
\end{lemma}

\begin{proof}
    For $x \in B$ put
    \[
        q_t(x) = (x^*x)(t)^{1/2}.
    \]
    Since every self-adjoint element of $B$ is central, the real-valued form
    \[
        (x,y)_t = \frac{1}{2} (x^*y + y^*x)(t)
    \]
    is well defined.  Positivity of $(sx + y)^*(sx + y)$ for $s\in\bbR$, followed by evaluation at $t$, gives the Cauchy--Schwarz inequality for $(\cdot,\cdot)_t$, and hence the triangle inequality for $q_t$.  Thus $q_t$ is a seminorm.

    The element $x$ is normal: writing $x = a + k$ with $a \in B_{\sa} = Z(B)_{\sa}$ and $k^* = -k$ gives $x^*x = xx^*$.  Consequently
    \[
        q_t(ax),\ q_t(xa) \leq \|a\|q_t(x) \qquad (a,x \in B),
    \]
    so $\ker q_t$ is a closed two-sided $*$-ideal.  It contains every central function vanishing at $t$, and hence $I_t \subseteq \ker q_t$.

    Conversely, suppose $(x^*x)(t) = 0$ and put $h = x^*x\in C(X,\bbR)_{+}$.  For $\varepsilon > 0$ let
    \[
        g_{\varepsilon} = \frac{h}{h + \varepsilon}\in C(X,\bbR).
    \]
    Then $g_{\varepsilon}(t) = 0$, so $g_{\varepsilon} x \in I_t$, and
    \[
        \|x - g_{\varepsilon} x\|^2 = \left\|\frac{\varepsilon^2h}{(h + \varepsilon)^2}\right\| \leq \frac{\varepsilon}{4}.
    \]
    Hence $x \in I_t$.  This proves the displayed identity.  Since $1\notin I_t$, the ideal is proper.  Finally, if $x \not= 0$, then $x^*x$ is a nonzero positive central function, so $(x^*x)(t) > 0$ for some $t$ and the image of $x$ in $B_t$ is nonzero.
\end{proof}

\begin{lemma}\label{lem:deriv-basic}
    Let $B$ be a semi-abelian real \AWstar-algebra and let $D$ be a derivation on $B$.  Then $D$ is bounded, $D$ vanishes on $B_{\sa}$, and $D$ carries skew-adjoint elements to skew-adjoint elements.  Moreover, for skew-adjoint $x,y$,
    \[
        \langle D(x),y\rangle + \langle x,D(y)\rangle = 0, \qquad \langle x,y\rangle = -\frac{1}{2} (xy + yx).
    \]
\end{lemma}

\begin{proof}
    The complex-linear extension of $D$ to the \cstar-complexification of $B$ is a derivation, hence is bounded by Sakai's theorem~\cite{Sak60}; therefore $D$ is bounded.  Every central projection $e$ satisfies
    \[
        D(e) = D(e^2) = D(e)e + eD(e) = 2eD(e),
    \]
    so $D(e) = 0$.  By Lemma~\ref{lem:semi-selfadjoint}, every self-adjoint element is a norm limit of finite real linear combinations of central projections.  Hence
    \[
        D|_{B_{\sa}} = 0.
    \]
    In particular,
    \[
        D(hx) = hD(x) \qquad (h \in Z(B)_{\sa}). \tag{7.2}
    \]
    
    Let $X$ be the Stone space of $Z(B)_{\sa} \cong C(X,\bbR)$ and let $I_t$ and $B_t$ be as in Lemma~\ref{lem:semi-fibers}.  Equation (7.2) and boundedness show that $D(I_t) \subseteq I_t$, so $D$ descends to a derivation $D_t$ on $B_t$.
    
    Every self-adjoint element of $B_t$ is scalar.  Hence, for $0 \not= x \in B_t$,
    \[
        x^*x = \lambda 1,\qquad xx^* = \mu 1
    \]
    with $\lambda = \mu = \|x\|^2 > 0$, so $x$ is invertible.  The real Gelfand--Mazur theorem gives
    \[
        B_t \cong \bbR,\ \bbC\ \text{or}\ \mathbb{H}.
    \]
    On $\bbR$ and $\bbC$ every real-linear derivation is zero.  On $\mathbb{H}$, if $u$ is imaginary then $u^2 \in \bbR 1$, so
    \[
        0 = D_t(u^2) = D_t(u)u + uD_t(u).
    \]
    Writing $D_t(u)$ as the sum of its real and imaginary parts shows that $D_t(u)$ is imaginary and orthogonal to $u$.  Applying the same identity to $u + v$ and polarizing yields
    \[
        \langle D_t(u),v\rangle + \langle u,D_t(v)\rangle = 0
    \]
    for imaginary $u,v$.  Thus on every fiber $D_t$ preserves the skew-adjoint part and is skew for the displayed form.
    
    The fibers separate points by Lemma~\ref{lem:semi-fibers}.  Therefore the two fiberwise identities lift to $B$, proving the result.
\end{proof}

For skew-adjoint $x,y \in B$ put
\[
    \langle x,y\rangle = -\frac{1}{2} (xy + yx)\in Z(B)_{\sa}.
\]
Then $\langle x,x\rangle = x^*x$, so the form is nondegenerate.

\begin{lemma}\label{lem:fiber-der}
    Let $B$ be semi-abelian and let $L : B \to B$ be a real-linear map such that
    \begin{enumerate}
        \item[(i)] $L|_{B_{\sa}} = 0$ and $L$ carries skew-adjoint elements to skew-adjoint elements;
        \item[(ii)] $L(hx) = hL(x)$ for $h \in Z(B)_{\sa}$;
        \item[(iii)] for skew-adjoint $x,y$,
        \[
            \langle Lx,y\rangle + \langle x,Ly\rangle = 0.
        \]
    \end{enumerate}
    Then $L$ is a derivation.
\end{lemma}

\begin{proof}
    First $L$ is automatically bounded.  If $x_n \to 0$ are skew-adjoint and $Lx_n \to y$, then for every skew-adjoint $w$,
    \[
        \langle y,w\rangle = \lim_n\langle Lx_n,w\rangle = -\lim_n\langle x_n,Lw\rangle = 0.
    \]
    Taking $w = y$ gives $y = 0$.  Since $L$ vanishes on the self-adjoint part, the closed graph theorem applies.
    
    Let $X$ be the Stone space of $Z(B)_{\sa} \cong C(X,\bbR)$ and let $I_t$ and $B_t$ be as in Lemma~\ref{lem:semi-fibers}.  Every self-adjoint element of $B_t$ is scalar, because it is represented by the self-adjoint part of a lift and $B_{\sa} = Z(B)_{\sa}$.  Hence every nonzero $x \in B_t$ satisfies
    \[
        x^*x = \lambda 1,\qquad xx^* = \mu 1,
    \]
    with $\lambda = \mu = \|x\|^2>0$, and is invertible.  The real Gelfand--Mazur theorem gives
    \[
        B_t \cong \bbR,\ \bbC\text{ or }\mathbb{H}.
    \]
    By (ii) and boundedness, $L$ descends to a linear map $L_t$ on every fiber and preserves the displayed skew form.  On $\bbR$ there is nothing to prove; on $\bbC$ the skew part is one-dimensional, so a skew-symmetric operator is zero; on $\mathbb{H}$ the skew part is the Euclidean space $\operatorname{Im}\mathbb{H} \cong \bbR^3$, and every skew-symmetric operator has the form $u \mapsto [a,u]$ for some imaginary quaternion $a$.  Thus every $L_t$ is a derivation.
    
    Therefore the derivation identity for $L$ holds in every central fiber.  Since the fibers separate points by Lemma~\ref{lem:semi-fibers}, the derivation defect vanishes in $B$.  Thus $L$ is a derivation.
\end{proof}

\begin{proposition}\label{prop:semi-local}
    Every local derivation and every 2-local derivation on a semi-abelian real \AWstar-algebra is a derivation.
\end{proposition}

\begin{proof}
    Let first $\Delta$ be a local derivation; by definition it is real-linear.  Lemma~\ref{lem:deriv-basic} applied to a local witness at each point gives $\Delta|_{B_{\sa}} = 0$ and shows that $\Delta$ preserves skew-adjointness.  For skew $x$, choose a derivation $D_x$ with $\Delta(x) = D_x(x)$.  Since $D_x$ is skew with respect to $\langle\cdot,\cdot\rangle$,
    \[
        \langle\Delta(x),x\rangle = 0.
    \]
    Polarizing and using linearity gives
    \[
        \langle\Delta(x),y\rangle + \langle x,\Delta(y)\rangle = 0. \tag{7.3}
    \]
    For $h \in Z(B)_{\sa}$ and skew $x,y$, (7.3) gives
    \[
        \langle\Delta(hx) - h\Delta(x),y\rangle = -h\langle x,\Delta(y)\rangle-h\langle\Delta(x),y\rangle = 0.
    \]
    The difference is skew-adjoint, so taking $y$ equal to that difference proves $\Delta(hx) = h\Delta(x)$.  Lemma~\ref{lem:fiber-der} now makes $\Delta$ a derivation.
    
    Now let $\Delta$ be 2-local.  It is homogeneous.  The same pointwise argument gives $\Delta(h) = 0$ for self-adjoint $h$ and shows that skew elements are sent to skew elements.  Given skew $x,y$, choose one derivation witnessing the pair; then (7.3) holds.  For skew $x,y,w$,
    \[
    \begin{aligned}
        \langle\Delta(x + y)-\Delta(x)-\Delta(y),w\rangle
        & = -\langle x + y,\Delta(w)\rangle + \langle x,\Delta(w)\rangle + \langle y,\Delta(w)\rangle\\
        & = 0.
    \end{aligned}
    \]
    Taking $w$ equal to the skew difference proves additivity on the skew part.  If $a = a^*$ and $k^* = -k$, a derivation witnessing the pair $(a + k,k)$ vanishes at $a$, so $\Delta(a + k) = \Delta(k)$.  Thus $\Delta$ is real-linear on $B$.  The same calculation as above gives $Z(B)_{\sa}$-linearity, and Lemma~\ref{lem:fiber-der} again applies.
\end{proof}

\begin{lemma}\label{lem:central-deriv}
    Let $e$ be a central projection in an algebra $A$.
        \begin{enumerate}
            \item[(i)] Every derivation $D$ satisfies $D(e) = 0$ and $D(eA) \subseteq eA$.
            \item[(ii)] Every 2-local derivation $\Delta$ on a real \AWstar-algebra satisfies $\Delta(ex) = e\Delta(x)$; its restriction to $eA$ is 2-local.
            \item[(iii)] Every linear local derivation preserves $eA$, and its restriction is local.
        \end{enumerate}
\end{lemma}

\begin{proof}
    For (i), $D(e) = D(e^2) = D(e)e + eD(e) = 2eD(e)$; multiplying by $e$ gives $eD(e) = 0$ and hence $D(e) = 0$.  The invariance follows from $D(ex) = eD(x)$.  For (ii), use a derivation witnessing the pair $(ex,x)$ and then restrict pairwise witnesses to $eA$.  For (iii), locality shows $\Delta(eA) \subseteq eA$; linearity and the central direct-sum decomposition give $\Delta(ex) = e\Delta(x)$.
\end{proof}

\begin{theorem}\label{thm:local-derivations}
    Every local derivation on a real \AWstar-algebra is a derivation, and every 2-local derivation on a real \AWstar-algebra is a derivation.
\end{theorem}

\begin{proof}
    Split $R = zR \oplus R_1$ by Lemma~\ref{lem:semi-split}.  Lemma~\ref{lem:central-deriv} preserves the two summands for both kinds of local map, and Proposition~\ref{prop:semi-local} handles $zR$.  The algebra $R_1$ has \AWstar-complexification by Theorem~\ref{thm:handelman}.

    The proofs of the local-derivation theorem of Karimov~\cite{Kar24} and the $2$-local theorem of Dadakhodjaev--Rakhimov~\cite{DR21} apply on the complementary central types.  We record a direct replacement for the finite type-$\mathrm{I}$ step, where quaternionic coefficient algebras must be allowed.

    Let $e$ be the maximal central projection of $R_1$ for which $eR_1$ is finite type $\mathrm{I}$.  The complexification of $eR_1$ is a finite type-$\mathrm{I}$ complex \AWstar-algebra and has its canonical faithful center-valued trace.  Conjugation equivariance lets this trace restrict to a faithful real-linear cyclic map
    \[
        T : eR_1 \longrightarrow Z(eR_1).
    \]
    Every derivation witness on $eR_1$ is inner: extend it complex-linearly to the complexification, use Olesen's theorem~\cite{Ole74}, and take the real part of an implementer exactly as in Proposition~\ref{prop:inner-complement} below.

    Suppose first that $\Delta$ is a real-linear local derivation on $eR_1$.  A derivation witnessing $\Delta$ at $x$ gives
    \[
        T(\Delta(x)x) = 0, \qquad T(\Delta(x)x^2) = 0. \tag{7.4}
    \]
    Polarizing the first identity yields
    \[
        T(\Delta(x)y) = -T(x\Delta(y)) \qquad (x,y \in eR_1). \tag{7.5}
    \]
    Taking the coefficient of $t$ in the second identity of (7.4), applied to $x + ty$, and then using (7.5), gives
    \[
        T\bigl((\Delta(x^2) - \Delta(x)x - x\Delta(x))y\bigr) = 0 \qquad (x,y \in eR_1).
    \]
    Taking $y$ to be the adjoint of the parenthesized defect and using faithfulness of $T$ gives
    \[
        \Delta(x^2) = \Delta(x)x + x\Delta(x).
    \]
    Thus $\Delta$ is a Jordan derivation.  Since a \cstar-algebra is semiprime, Bre\v{s}ar's theorem makes $\Delta$ a derivation~\cite{Bre88}.

    Now let $\Delta$ be $2$-local on $eR_1$.  A common derivation witness for a pair $(x,y)$ gives (7.5) without any prior additivity.  Hence, for arbitrary $z$,
    \[
    \begin{aligned}
        T\bigl((\Delta(x + y) - \Delta(x) - \Delta(y))z\bigr)
        & = -T((x + y)\Delta(z)) + T(x\Delta(z)) + T(y\Delta(z))\\
        & = 0.
    \end{aligned}
    \]
    Taking $z$ to be the adjoint of the displayed difference and using faithfulness proves additivity.  Homogeneity follows directly from $2$-locality.  A witness for the pair $(x,x^2)$ gives the Jordan identity, and Bre\v{s}ar's theorem again gives a derivation.

    Thus both kinds of local map are derivations on the finite type-$\mathrm{I}$ band and, by the cited arguments, on its complement.  Recombining the central summands proves the theorem.
\end{proof}

The conclusion cannot be strengthened to innerness.

\begin{proposition}\label{prop:inner-complement}
    Let $z$ be the maximal central semi-abelian projection of a real \AWstar-algebra $R$.  Every derivation on $(1 - z)R$ is inner.  Equivalently, every derivation on $R$ has the form
    \[
        D = D_0 \oplus \operatorname{ad}a
    \]
    with $D_0$ a derivation on $zR$ and $a \in(1 - z)R$.  In general $D_0$ need not be inner.
\end{proposition}

\begin{proof}
    By Theorem~\ref{thm:handelman}, $((1 - z)R)_{\bbC}$ is complex \AWstar.  A derivation $D_1$ on $(1 - z)R$ extends complex-linearly to a derivation on the complexification.  Olesen's theorem~\cite{Ole74} gives an implementer $c = a + ib$.  For $x \in (1 - z)R$,
    \[
        D_1(x) = [a,x] + i[b,x]
    \]
    lies in the real form, so $[b,x] = 0$ for every $x$ and $D_1 = [a,\cdot]$.  The quaternionic sequence algebra in Proposition~\ref{prop:quat-example} supplies an outer derivation on a semi-abelian algebra.
\end{proof}

Thus the innerness assertion is false, but the derivation-only parts of the local and 2-local statements are valid without any complexification hypothesis.

\end{document}